\documentclass{aic}
\usepackage[utf8]{inputenc}
\usepackage{amsthm,amssymb}
\usepackage{amsmath,nicefrac,amsfonts}
\usepackage{thmtools,mathtools}
\usepackage{caption,subcaption}
\usepackage{epsfig,graphicx,graphics,color,tikz,tikz-cd,tcolorbox}
\usepackage{enumerate,enumitem}
\usepackage[sort,nocompress]{cite}
\usepackage{xspace}
\usepackage{bbm}
\usepackage{hyperref}
\usepackage[capitalise]{cleveref}

\newlength{\bibitemsep}
\newlength{\bibparskip}
\let\oldthebibliography\thebibliography
\renewcommand\thebibliography[1]{%
  \oldthebibliography{#1}%
  \setlength{\parskip}{\bibitemsep}%
  \setlength{\itemsep}{\bibparskip}%
}

\makeatletter
\renewcommand{\paragraph}{%
  \@startsection{paragraph}{4}%
  {\z@}{1.2ex \@plus 1ex \@minus .2ex}{-0.5em}%
  {\normalfont\normalsize\bfseries}%
}
\makeatother

\theoremstyle{plain}
\newtheorem{thm}{Theorem}[section]
\newtheorem{lem}[thm]{Lemma}
\newtheorem{cor}[thm]{Corollary}
\newtheorem{cl}[thm]{Claim}
\newtheorem{prop}[thm]{Proposition}

\theoremstyle{definition}

\newtheorem{rem}[thm]{Remark}

\DeclareMathOperator\xc{xc}
\DeclareMathOperator\fr{frac}
\DeclareMathOperator\bmm{bmm}

\newcommand{\bR}{\mathbb{R}}

\newcommand{\bZ}{\mathbb{Z}}

\newcommand{\bF}{\mathbb{F}}

\newcommand{\fc}{\mathfrak{c}}

\newcommand{\cB}{\mathcal{B}}
\newcommand{\cC}{\mathcal{C}}

\newcommand{\cH}{\mathcal{H}}

\newcommand{\cQ}{\mathcal{Q}}
\newcommand{\cR}{\mathcal{R}}

\newcommand{\cA}{\mathcal{A}}

\newcommand*\diff{\mathop{}\!\mathrm{d}}

\let\Right\bigr
\let\Left\bigl
\makeatletter
\def\bigr#1{\Right#1\@ifnextchar){\!\bigr}{}}
\def\bigl#1{\Left#1\@ifnextchar({\!\bigl}{}}
\makeatother

\aicAUTHORdetails{%
    title = {Matroid Products via Submodular Coupling},
    author = {Krist\'of B\'erczi, Bogl\'arka Geh\'er, Andr\'as Imolay, L\'aszl\'o Lov\'asz, Bal\'azs Maga, and Tam\'as Schwarcz},
    plaintextauthor = {Kristof Berczi, Boglarka Geher, Andras Imolay, Laszlo Lovasz, Balazs Maga, Tamas Schwarcz},
    runningauthor = {K. B\'erczi, B. Geh\'er, A. Imolay, L. Lov\'asz, B. Maga, and T. Schwarcz},
    keywords = {amalgams, coupling, coverage functions, matroids, product, quasi-product, quotients, submodular functions, tensor product},
}   

\aicEDITORdetails{%
   year={2026},
   number={11},
   received={25 June 2025},   
   published={27 November 2026},  
   doi={10.19086/aic.2026.11.},      
}   

\begin{document}

\begin{frontmatter}[classification=text]

\title{Matroid Products via Submodular Coupling\titlefootnote{An extended abstract~\cite{berczi2025matroid} of the paper appeared at the 57th Annual ACM Symposium on Theory of Computing (STOC 2025).}}

\author[kristof]{Krist\'of B\'erczi}
\author[bogi]{Bogl\'arka Geh\'er}
\author[andris]{Andr\'as Imolay}
\author[laci]{L\'aszl\'o Lov\'asz}
\author[balazs]{Bal\'azs Maga}
\author[tamas]{Tam\'as Schwarcz}

\begin{abstract}
The study of matroid products traces back to the 1970s, when Lovász and Mason studied the existence of various types of matroid products with different strengths. Among these, the tensor product is arguably the most important, and can be considered as an extension of the tensor product from linear algebra. However, Las Vergnas showed that the tensor product of two matroids does not always exist. Over the following four decades, matroid products remained surprisingly underexplored, regaining attention only in recent years due to applications in tropical geometry and the limit theory of matroids.

In this paper, inspired by the concept of coupling in probability theory, we introduce the notion of coupling for matroids -- or, more generally, for submodular set functions. This operation can be viewed as a relaxation of the tensor product. We prove, however, that, unlike a tensor product, a coupling always exists for any two nonnegative submodular functions vanishing on the empty set, and can be chosen to be increasing if the original functions are increasing. As a corollary, we show that two matroids always admit a matroid coupling, leading to a novel operation on matroids. Our construction is algorithmic, providing an oracle for the coupling matroid through a polynomial number of oracle calls to the original matroids.

We apply this construction to derive new necessary conditions for matroid representability and establish a connection between tensor products and Ingleton's inequality. Additionally, we verify the existence of set functions that are universal with respect to a given property, meaning that any set function over a finite domain with that property can be obtained as a quotient.
\end{abstract}
\end{frontmatter}

\newpage
\tableofcontents
\newpage

\section{Introduction}
\label{sec:intro}

Coupling of probability measures, generally attributed to Doeblin~\cite{doeblin1938expose}, is a key concept in probability theory through which random variables can be compared with each other. For two probability measures, a {\it coupling} is a joint probability measure over the product of the underlying spaces, such that the marginals correspond to the given probability measures. Formally, if $\mu_1$ and $\mu_2$ are probability measures on spaces $(S_1, \cB_1)$ and $(S_2, \cB_2)$, respectively, a coupling is a joint probability measure $\mu$ on the product space $(S_1\times S_2, \cB_1 \otimes \cB_2)$ such that the marginal of $\mu$ on $S_i$ is $\mu_i$ for $i=1,2$, i.e., $\mu(X_1\times S_2)=\mu_1(X_1)$ for any $X_1\in \cB_1$ and $\mu(S_1\times X_2)=\mu_2(X_2)$ for any $X_2\in \cB_2$. Usually, the goal is to choose a coupling that allows for comparisons between the probability measures, such as minimizing differences or distances between them. Coupling techniques have found applications in various areas. In optimal transport theory, the goal is to find the most efficient way to transform one distribution into another, with applications in economics, statistical physics, and machine learning~\cite{chewi2025statistical}. In stochastic processes, coupling helps to describe the joint behavior of random variables and to analyze and bound convergence rates in chains with more complex dependencies~\cite{levin2017markov,aldous-fill-2014}.

Interestingly, a similar notion has also appeared in a combinatorial line of research on matroids. Let $M_1=(S_1,r_1)$ and $M_2=(S_2,r_2)$ be finite matroids over ground sets $S_1$ and $S_2$ with rank functions $r_1$ and $r_2$, respectively. A matroid $M=(S_1\times S_2,r)$ is called a {\it quasi-product} of $M_1$ and $M_2$ if the restriction of $M$ to $\{e_1\}\times S_2$ is isomorphic to $M_2$ by the natural bijection with $S_2$ for every non-loop $e_1\in S_1$ and is the rank-zero matroid when $e_1$ is a loop, and analogously with the two factors interchanged.
If furthermore $Y_1\times Y_2$ is a flat of $M$ for all flats $Y_1$ of $M_1$ and $Y_2$ of $M_2$, then $M$ is called a {\it product} of $M_1$ and $M_2$. Finally, if $r(Y_1\times Y_2)=r_1(Y_1)\cdot r_2(Y_2)$ holds for every $Y_1\subseteq S_1,Y_2\subseteq S_2$, then $M$ is a {\it tensor product} of $M_1$ and $M_2$. Note that at this point, it is not clear whether any two matroids admit a tensor product, a product, or even a quasi-product. Mason~\cite{mason1977geometric} and Lovász~\cite{lovasz1977flats} provided constructions yielding a product of rank $r_1(S_1)+r_2(S_2)-1$ via Dilworth truncation. The existence of a tensor product, however, was asked as an open question in~\cite{lovasz1977flats}. Note that if both $M_1$ and $M_2$ are linear matroids over the same field, then it is easy to construct a tensor product by taking the tensor product of the matrices representing them. Nevertheless, Las Vergnas~\cite{las1981products} later showed that the situation is generally not so fortunate: for example, the Vámos matroid and the rank-2 uniform matroid on three elements do not admit a tensor product. 

One can think of the tensor product of two matroids as a construction where the product of independent sets is itself independent. However, since a tensor product does not exist for every pair of matroids, Mason~\cite{mason1981glueing} raised the question of whether a freest product exists among the possible products. Later, Las Vergnas~\cite{las1981products} answered this question in the negative~\cite{lovasz1977flats,mason1981glueing}.

\subsection{Our results and techniques}
\label{sec:our}

In this paper, we extend the concept of coupling to matroids and, more broadly, to submodular functions. We prove that, unlike tensor products, couplings always exist, leading to a novel operation on matroids. As an application, we establish the existence of functions that are universal with respect to some property, meaning that any set function with that property can be obtained as a quotient. This result brings new insights into the developing limit theory of matroids and submodular functions.

Let $\varphi_1\colon 2^{S_1}\to\bR$ and $\varphi_2\colon2^{S_2}\to\bR$ be set functions defined over finite ground sets $S_1$ and $S_2$, respectively. The notion of a tensor product can be naturally extended to set functions by calling a function $\varphi\colon 2^{S_1\times S_2}\to\bR$ a tensor product of $\varphi_1$ and $\varphi_2$ if $\varphi(X_1\times X_2)=\varphi_1(X_1)\cdot\varphi_2(X_2)$ holds for every $X_1\subseteq S_1$ and $X_2\subseteq S_2$. As a new concept, we call $\varphi$ a {\it coupling} of $\varphi_1$ and $\varphi_2$ if $\varphi(X_1\times S_2)=\varphi_1(X_1)\cdot \varphi_2(S_2)$ for every $X_1\subseteq S_1$ and $\varphi(S_1\times X_2)=\varphi_1(S_1)\cdot \varphi_2(X_2)$ for every $X_2\subseteq S_2$. In other words, a coupling is a set function on the product set, where the projections onto each coordinate return the corresponding $\varphi_i$, up to a constant multiplier. It is worth emphasizing that neither a tensor product nor a coupling is uniquely determined by the factors $\varphi_1$ and $\varphi_2$.

Motivated by the goal of defining a coupling of matroids, in Section~\ref{sec:pairs} we focus on submodular functions. First, we show that any two nonnegative submodular functions vanishing on the empty set admit a submodular coupling (Theorem~\ref{thm:submod}). The proof provides an explicit formula for the coupling function using two arbitrary modular functions. However, the proof itself does not imply that if both $\varphi_1$ and $\varphi_2$ are monotonically increasing, then the coupling function $\varphi$ is also monotonically increasing. Since this property is essential for the coupling of two matroids to result in a matroid, next we show that a $k_1$-polymatroid function and a $k_2$-polymatroid function have a $(k_1\cdot k_2)$-polymatroid coupling, which is also integer-valued if the original functions are (Theorem~\ref{thm:poly}). The proof follows an idea similar to that of the submodular case, but the modular functions in question need to be chosen as elements of the base polyhedra of $\varphi_1$ and $\varphi_2$, respectively. As a corollary, we get that any two matroids admit a coupling (Corollary~\ref{cor:matroid}). We provide several equivalent characterizations of pairs of matroids having a coupling that is also a tensor product (Theorem~\ref{thm:tensor}), and establish a strong connection between tensor products and Ingleton's inequality, a fundamental tool in the theory of representable matroids (Theorem~\ref{thm:tensor_ingleton}). As a corollary, we get a new proof of the fact that the uniform matroid $U_{2,3}$ and the Vámos matroid do not admit a tensor product. We separately study a subclass of increasing submodular functions with strong structural properties, called ``coverage functions'' in combinatorial optimization. These functions were originally introduced by Choquet for analytical studies~\cite{choquet1954theory}. He defined them in terms of a sequence of inequalities strengthening submodularity, and proved that this definition is equivalent to what is now the combinatorial definition. We prove that a coverage function and an increasing submodular function have an increasing submodular tensor product (Corollary~\ref{cor:inftens}), and that any two coverage functions have a tensor product that is a coverage function (Corollary~\ref{cor:inf}). The proofs rely on the characterization of the extreme rays of the cone of coverage functions, as established by Choquet. Finally, we describe how to extend our results on submodular functions defined on finite ground sets to those with infinite domains (Theorems~\ref{thm:inf1} and \ref{thm:inf2}).

Let $\varphi$ be a set function over some finite ground set $S$, and let $\cQ=(Q_1,\dots,Q_q)$ be a partition of $S$ into $q$ possibly empty parts. The {\it quotient} of $\varphi$ with respect to $\cQ$ is a set function $\varphi/\cQ$ over $[q]$ defined by $(\varphi/\cQ)(X)=\varphi(\bigcup_{i\in X} Q_i)$ for $X\subseteq [q]$.
In particular, if $\varphi\colon 2^{S_1\times S_2}\to\bR$ is a coupling of $\varphi_1\colon 2^{S_1}\to\bR$ and $\varphi_2\colon2^{S_2}\to\bR$, and $\cQ_1$ and $\cQ_2$ are the partitions of $S_1\times S_2$ into fibers of the form $\{e_1\}\times S_2$ and $S_1\times \{e_2\}$, respectively, where $e_1\in S_1$ and $e_2\in S_2$, then $\varphi_1$ and $\varphi_2$ are the rescalings of $\varphi/\cQ_1$ and $\varphi/\cQ_2$ by $1/\varphi_2(S_2)$ and $1/\varphi_1(S_1)$, respectively.
The positive results of Section~\ref{sec:pairs} motivate the following question: Given a property of set functions, is there a function that is universal with respect to that property, in the sense that every function exhibiting the property can be obtained as its quotient? 

In Section~\ref{sec:universal}, we examine this problem and show that, under some natural assumptions, the answer is positive (Theorem~\ref{thm:countable}). The proof is analytic, and the result has far-reaching implications for functions defined on finite ground sets as well. For example, we obtain that there exists a submodular function $\varphi\colon 2^\bZ\to[0,1]$ such that any nonnegative increasing normalized submodular function is a quotient of $\varphi$ (Corollary~\ref{cor:univ_increasing_submod}), and we prove an analogous result for coverage functions as well (Corollary~\ref{cor:infalt}). We conclude the section by providing an explicit construction for a universal coverage function, in sharp contrast with the preceding non-constructive universality results; the formula of the function is surprisingly simple (Theorem~\ref{thm:construction}).

It is worth emphasizing that for functions over finite ground sets, our proofs are algorithmic in the sense that given value oracles for $\varphi_1$ and $\varphi_2$, one can construct a value oracle for their coupling $\varphi$ with the desired properties in polynomial time.

\subsection{Related work and motivation}
\label{sec:motiv}

Understanding how matroids can be combined not only helps in studying properties of larger systems formed from simpler matroid components, but is also closely related to a range of problems and techniques in combinatorial optimization. In what follows, we give a brief overview of related topics. 

\paragraph{Tropical geometry.}
Tensor products of matroids have deep connections to tropical geometry. A key consequence of Las Vergnas's counterexample is that the tensor product of tropical linear spaces does not always yield a tropical linear space; see e.g.~\cite{giansiracusa2020matroidal}. The notion of tensor products can be extended to more than two matroids. If all $k$ matroids in the product are identical to a given matroid $M=(S,r)$, the tensor product defines a matroid structure on the set of ordered $k$-tuples from $S$, known as the {\it $k$-th power} of $M$. For unordered $k$-tuples, we consider {\it symmetric tensors}, which remain invariant under any permutation of their variables. The concept of symmetric powers of matroids was first introduced by Lovász~\cite{lovasz1977flats} and Mason~\cite{mason1981glueing}, who noted that not all matroids admit higher symmetric powers.
Draisma and Rincón~\cite{draisma2021tropical} established a link between tropical ideals and matroid symmetric powers by showing that the Bergman fan of the direct sum of the Vámos matroid and $U_{2,3}$ is not a tropically realizable variety. In~\cite{anderson2024matroid}, Anderson proved an equivalence between valuated matroids with arbitrarily large symmetric powers and tropical linear spaces represented as varieties of tropical ideals. Brakensiek, Dhar, Gao, Gopi, and Larson~\cite{brakensiek2024rigidity} explored the connection between rigidity matroids of graphs and matroids arising from linear-algebraic constructions like tensor products and symmetric products. Matroid tensor products in tropical geometry were also examined in~\cite{khan2024tropical, eur2024cohomologies}.

\paragraph{Linear matroids.}
A matroid is called linear if its independent sets can be represented by vectors in a vector space over some field $\bF$. Such matroids are particularly interesting because they allow for a rich connection between linear algebra and combinatorial structures. 
Regularity of a matroid, i.e., representability over every field, can be tested algorithmically using Seymour's decomposition theorem~\cite{seymour1980decomposition},
since such decompositions can be found efficiently~\cite{truemper1982efficiency}.
However, Truemper~\cite{truemper1982efficiency} showed that many representability questions cannot be efficiently solved. These include deciding representability over a specific field, over all fields with a given characteristic, and, most importantly, over any field. Since determining representability is oracle-hard in general, the focus of research has shifted toward finding necessary or sufficient conditions. Ingleton's inequality~\cite{ingleton1971representation} is one of the most prominent examples, giving a necessary condition for a matroid to be linear. Surprisingly, the existence of a tensor product with $U_{2,3}$ gives another necessary condition for representability, and this condition implies Ingleton's inequality.

\paragraph{Limits of matroids.} 
The limit theory of graphs provides powerful tools for analyzing sequences of graphs and their structural similarities through analytic methods. By defining convergence in terms of distributions of small subgraphs (left-convergence) and homomorphisms into small graphs (right-convergence), it aids in understanding complex networks and their applications in various fields of mathematics and computer science; we refer the interested reader to~\cite{lovasz2012large} for a thorough introduction. In~\cite{berczi2026quotient}, a new form of right-convergence called {\it quotient-convergence} was introduced for set functions, which eventually led to a notion of convergence of matroids through their rank functions. The limit object of such a sequence is a submodular function~\cite{lovasz2023submodular}. One of the main research subjects is understanding the structure and properties of these limit objects, which heavily relies on examining their quotients and the coupling of the functions in the sequence.

\paragraph{Extended formulations.} 
For a polytope $P\subseteq\bR^d$, an extension of $P$ is a polytope $Q\subseteq\bR^{d'}$ such that $P=\pi(Q)$ for some affine map $\pi\colon\bR^{d'}\to\bR^d$. The {\it extension complexity} $\xc(P)$ of $P$ is defined as the minimum number of facets of an extension of $P$. If $Q$ is given by a linear description $Q=\{y\in\bR^{d'}\mid Ay\le b\}$, then $\{x\in\bR^{d}\mid Ay\le b,x=\pi(y)\}$ is called an {\it extended formulation} of $P$. In this terminology, the extension complexity of a polytope is the minimum number of inequality constraints in an extended formulation. Since the fundamental work of Yannakakis~\cite{yannakakis1991expressing}, the extension complexity has been settled for various families of combinatorial polytopes, such as the traveling salesman and cut polytopes~\cite{fiorini2012linear}, the stable set polytope~\cite{fiorini2012linear,goos2018extension}, the matching polytope~\cite{rothvoss2017matching}, and recently the independence polytope of regular matroids~\cite{aprile2022regular}. Taking the quotient of a function is similar to projection; however, rather than retaining only certain coordinates of a solution, it partitions the coordinates and sums them within each partition class.  Consequently, the coupling of matroids, or more generally of submodular functions, leads to a different type of ``extended'' formulation that is of independent combinatorial interest.

\subsection{Organization}
\label{sec:structure}

The rest of the paper is organized as follows. In Section~\ref{sec:prelim}, we introduce basic definitions, notation, and relevant results on submodular functions, matroids, and polymatroids. Section~\ref{sec:pairs} is devoted to verifying the existence of couplings for various functions, first focusing on submodular functions in Section~\ref{sec:submod} and then extending these results to polymatroid functions in Section~\ref{sec:polymatroid}. In Section~\ref{sec:matroids}, we show how these observations lead to one of the main results of the paper: the existence of matroidal couplings that are almost tensor products. Section~\ref{sec:local} explores the relation between our construction and matroidal amalgams. To understand the fine line between couplings and tensor products, we establish a necessary and sufficient condition for a coupling to be a tensor product in Section~\ref{sec:tensor}, and in Section~\ref{sec:ingleton}, we provide new necessary conditions for matroid representability. Section~\ref{sec:infinite} studies coverage and $k$-alternating functions, and Section~\ref{sec:infground} extends the results to infinite ground sets. In Section~\ref{sec:universal}, we use the existence of couplings to show that there are functions that are universal with respect to certain properties. The proof is analytic and hence not constructive; we encourage first-time readers to skip the technical parts of Sections~\ref{sec:prep} and~\ref{sec:existence}. As an application of the abstract existence theorem, we deduce that universal functions exist for several set function properties; these examples can be found in Section~\ref{sec:applications}. Finally, we give an explicit construction for a universal coverage function in Section~\ref{sec:infalt}.

\section{Preliminaries}
\label{sec:prelim}

\paragraph{Basic notation.} 
We denote the sets of {\it reals} and {\it integers} by $\bR$ and $\bZ$, respectively, and add $\ge 0$ or $>0$ as a subscript when considering {\it nonnegative} or {\it strictly positive} values only. The {\it cardinality of $\bR$} is denoted by $\fc$. For a real number $a\in\bR$, we denote its {\it fractional part} by $\fr(a)$. For a positive integer $k$, we use $[k]\coloneqq \{1,\dots,k\}$ while $[0] = \emptyset$ by convention. Given a set $S$, a subset $X\subseteq S$, and $y\in S$, the sets $X\setminus \{y\}$ and $X\cup \{y\}$ are abbreviated as $X-y$ and $X+y$, respectively. Moreover, we often denote a singleton set $\{x\}$ by $x$ when this causes no confusion. The {\it complement} of $X$ is denoted by $X^c=S\setminus X$. Given a function $\mu\colon S\to\bR_{\ge 0}$, we use the notation $\mu(X)\coloneqq\sum_{x\in X}\mu(x)$. If $S=S_1\times S_2$ for some sets $S_1$ and $S_2$, then for any $Z\subseteq S$, $e_1\in S_1$, and $e_2\in S_2$, we use $Z_{e_1}=Z\cap (\{e_1\}\times S_2)$ and $Z^{e_2}=Z\cap (S_1\times \{e_2\})$ to denote the {\it $e_1$-} and {\it $e_2$-fibers} of $Z$, respectively. Furthermore, we use $\pi_i\colon 2^S\to 2^{S_i}$ to denote the {\it coordinate maps}, that is, $\pi_1(Z)=\{e_1\in S_1\mid Z_{e_1}\neq\emptyset\}$ and $\pi_2(Z)=\{e_2\in S_2\mid Z^{e_2}\neq\emptyset\}$, respectively.

\paragraph{Matroids.}
For basic definitions concerning matroids, we refer the reader to~\cite{oxley2011matroid}. A {\it matroid} $M=(S,r)$ is defined by its finite {\it ground set} $S$ and its {\it rank function} $r\colon2^S\to\bZ_{\ge 0}$ that satisfies the {\it rank axioms}: for all $X, Y \subseteq S$,
\begin{enumerate}[label=(R\arabic*), leftmargin=3em] \itemsep0em
\item $r(\emptyset)=0$,
\item $X\subseteq Y\Rightarrow r(X)\le r(Y)$,
\item $r(X)\le |X|$,
\item $r(X)+r(Y)\ge r(X\cap Y)+r(X\cup Y)$.
\end{enumerate}
Here, (R1) is a standard normalizing assumption, while (R2) requires the rank function to be increasing, (R3) its subcardinality, and (R4) its submodularity. The {\it rank} of the matroid is $r(S)$, and by the {\it normalized rank function} of $M$ we mean the set function $r/r(S)$. A subset $X\subseteq S$ is called {\it independent} if $|X|=r(X)$. An inclusionwise minimal non-independent set is a {\it circuit}, while a {\it loop} is a circuit consisting of a single element. A {\it flat} is a set $F\subseteq S$ such that $r(F+e)=r(F)+1$ holds for every $e\in S\setminus F$. A matroid is called {\it modular} if $r(F_1)+r(F_2)=r(F_1\cap F_2)+r(F_1\cup F_2)$ for any pair of flats $F_1,F_2$. For sets $X,Y\subseteq S$, we say that $X$ {\it spans} $Y$ in $M$ if $r(X\cup Y)=r(X)$. 

Given a subset $S'\subseteq S$, the {\it restriction} of $M$ to $S'$ is a matroid $M|S'=(S',r')$ whose rank function is the restriction of $r$ to subsets of $S'$. The {\it contraction of $S'$} results in a matroid $M/S'=(S\setminus S',r')$ with rank function $r'(X)=r(X\cup S')-r(S')$ for each $X\subseteq S\setminus S'$. A matroid $N$ that can be obtained from $M$ by a sequence of restrictions and contractions is called a {\it minor} of $M$. The {\it direct sum} $M_1\oplus M_2$ of matroids $M_1=(S_1,r_1)$ and $M_2=(S_2,r_2)$ on disjoint ground sets is the matroid $M=(S_1\cup S_2,r)$ whose independent sets are the disjoint unions of an independent set of $M_1$ and an independent set of $M_2$, that is, $r(X)=r_1(X\cap S_1)+r_2(X\cap S_2)$ for $X\subseteq S_1\cup S_2$. If $M=(S,r)$ and $B$ is a finite index set, then $\oplus_{e\in B} M$ denotes the direct sum of labeled copies of $M$, where the copy indexed by $e\in B$ is identified with the matroid on $S\times\{e\}$ naturally isomorphic to $M$.

A matroid is {\it linear} or {\it representable} over some field $\bF$ if there exists a family of vectors from a vector space over $\bF$ whose linear independence relation is the same as the independence relation of the matroid. If the matroid is representable over every field, then it is called {\it regular}. A {\it uniform matroid} of rank $r$ over a ground set of size $n$ has rank function $r(X)=\min\{|X|,r\}$ and is denoted by $U_{r,n}$. 
Given pairwise disjoint sets $S_1, \dots, S_q\subseteq S$, the corresponding {\it partition matroid} $M=(S,r)$ is defined by setting $r(X)=|\{i\mid X\cap S_i\neq\emptyset\}|$ for all $X\subseteq S$\footnote{In the literature, partition matroids are often defined in a more general form where $r(X)=\sum_{i=1}^q\min\{|X\cap S_i|,g_i\}$ for some $g_1,\dots,g_q\in\bZ_{\ge 0}$. However, in this paper, we are interested only in the case when $g_i=1$ for all $i\in[q]$ except for at most one, which is $0$.}. It is not difficult to see that both uniform and partition matroids are representable over the reals.

\paragraph{Submodular and polymatroid functions.}
By $S$, we always denote a finite ground set. A set function $\varphi \colon 2^S \to \bR$ is {\it normalized} if $\varphi(\emptyset)=0$ and $\varphi(S)=1$. We say that $\varphi$ is {\it increasing} if $X\subseteq Y$ implies $\varphi(X)\le\varphi(Y)$, and {\it decreasing} if $X\subseteq Y$ implies $\varphi(X)\ge\varphi(Y)$. The function $\varphi$ is {\it submodular} if 
\begin{equation*}
\varphi(X)+\varphi(Y) \ge \varphi(X \cap Y)+\varphi(X \cup Y)    
\end{equation*} 
for all $X,Y\subseteq S$, {\it supermodular} if $-\varphi$ is submodular, and {\it modular} if it is both sub- and supermodular. 

A function $\varphi\colon 2^S\to\bR_{\ge 0}$ is called a {\it polymatroid function} if it is an increasing, submodular function with $\varphi(\emptyset)=0$. Furthermore, given $k \in \bR_{\ge 0}$, we say that a polymatroid function $\varphi$ is a {\it $k$-polymatroid function} if $\varphi(X) \le k \cdot |X|$ holds for every subset $X \subseteq S$. Observe that polymatroid functions retain three of the four basic properties of matroid rank functions: (R1), (R2), and (R4), omitting only the subcardinality and integrality requirements. Thus, matroid rank functions form a subclass of polymatroid functions. As a sort of reverse statement, Helgason~\cite{helgason2006aspects} showed that every integer-valued polymatroid function can be obtained as the quotient of a matroid rank function. 

\begin{prop}[Helgason]\label{prop:helgason}
If $\varphi\colon 2^S\to\bZ_{\ge 0}$ is an integer-valued polymatroid function over a ground set $S$, then there exist a matroid $M = (S', r)$ and a mapping $\theta\colon S' \to S$ such that $\varphi(X) = r(\theta^{-1}(X))$ for all $X \subseteq S$.
\end{prop}

The {\it base polyhedron} of a nonnegative submodular function $\varphi\colon 2^S\to \bR_{\ge 0}$ is defined as \[B(\varphi)\coloneqq \{x \in \bR^S_{\ge 0} \mid x(S) = \varphi(S),~
x(Z) \le \varphi(Z) \text{ for every $Z\subseteq S$}\}.\]
The following is a fundamental result in polyhedral combinatorics; see e.g.~\cite[Corollary~14.2.3]{frank2011connections}.

\begin{prop}\label{prop:empty}
If $\varphi$ is a nonnegative submodular function, then $B(\varphi)$ is nonempty. Furthermore, if $\varphi$ is integer-valued, then the vertices of $B(\varphi)$ are integer.
\end{prop}

\paragraph{Coverage functions.}
A {\it coverage function} $\varphi\colon 2^S\to\bR_{\ge 0}$ is defined by a bipartite graph $G=(S,T;E)$ and weights $w\in\bR^T_{\ge 0}$ by setting $\varphi(X)=\sum_{t\in N(X)}w_t$ for $X\subseteq S$, where $N(X)$ denotes the set of neighbors of $X$ in $G$. 
Coverage functions also have an alternative characterization involving inequalities~\cite{bgils2026decomposition}: $\varphi$ is a coverage function if and only if $\varphi(\emptyset)=0$ and, for any $k\in\bZ_{>0}$,
\begin{equation}
\sum_{K \subseteq [k]} (-1)^{|K|}\varphi\bigl(A_0 \cup \bigcup_{i \in K} A_i\bigr)\le 0 \tag*{\textsc{($k$-Alt)}}\label{eq:kalt}
\end{equation}
holds for any choice of subsets $A_0, A_1, \ldots, A_k \subseteq S$. It is worth noting that for $k = 1$ and $k = 2$, inequality \ref{eq:kalt} is equivalent to the property that $\varphi$ is increasing and increasing submodular, respectively. In general, for $k\ge 3$, the inequality describes a strengthened form of increasing submodularity\footnote{In~\cite{bgils2026decomposition}, it was shown that for functions over a finite domain, the class of coverage functions coincides with that of infinite-alternating functions, an important subclass of submodular functions that was studied by Choquet~\cite{choquet1954theory} in the analytic setting.}.  For any nonempty $A\subseteq S$, let $\varphi_A \colon 2^S \to \bR$ be defined as
\begin{equation} 
   \varphi_A(X)=\begin{cases}
    0 &\text{if $X \cap A=\emptyset$,} \\
    1 &\text{if $X \cap A\neq\emptyset$.}
    \end{cases}
    \label{eq:extremal_def}
\end{equation}
It is not difficult to see that $\varphi_A$ is a normalized coverage function. Choquet~\cite[Section 43]{choquet1954theory} showed that these set functions correspond to the extreme rays of the convex cone of coverage functions; see also~\cite{lovasz2023submodular} for a different proof.

\begin{prop}[Choquet]\label{prop:extremal}
    The set of normalized extremal elements of the convex cone of coverage functions over a finite set $S$ is $\{\varphi_A\mid \emptyset\neq A\subseteq S\}$.
\end{prop}

For further details on coverage functions, we recommend~\cite{choquet1954theory,lovasz2023submodular,bgils2026decomposition}.

\section{Coupling submodular functions}
\label{sec:pairs}

In this section we show that two matroids always admit a coupling, even though they need not admit a tensor product.
We first prove this statement for nonnegative submodular functions in Section~\ref{sec:submod} and then extend the discussion to polymatroid functions in Section~\ref{sec:polymatroid}. Matroids and the relation between coupling and tensor products are addressed in Section~\ref{sec:matroids}. We study $k$-alternating and coverage functions in Section~\ref{sec:infinite}. Finally, in Section~\ref{sec:infground} we explain how all these results can be extended to functions over an infinite ground set.

\subsection{Submodular functions}
\label{sec:submod}

As a first step, we verify that couplings exist for submodular functions. Let $\varphi_1\colon 2^{S_1} \to \bR_{\ge 0}$ and $\varphi_2 \colon 2^{S_2} \to \bR_{\ge 0}$ be submodular functions over ground sets $S_1$ and $S_2$, respectively, and let $S = S_1\times S_2$. Recall that for any $Z\subseteq S$, $e_1\in S_1$, and $e_2\in S_2$,  we use $Z_{e_1}=Z\cap (\{e_1\}\times S_2)$ and $Z^{e_2}=Z\cap (S_1\times \{e_2\})$ to denote the $e_1$- and $e_2$-fibers of $Z$, respectively. For any $\mu_1\colon S_1 \to \bR_{\ge 0}$ and $\mu_2\colon S_2 \to \bR_{\ge 0}$, define the set function $b$ on $S$ by

\begin{equation} \label{eq:submod_coupling}
    b(Z) \coloneqq\sum_{e_1\in S_1}\mu_1(e_1) \cdot \varphi_2(\pi_2(Z_{e_1}))+\sum_{e_2\in S_2}\mu_2(e_2)\cdot  \varphi_1(\pi_1(Z^{e_2})) -\sum_{(e_1,e_2)\in Z} \mu_1(e_1) \cdot \mu_2(e_2).
\end{equation}
Note that, by the nonnegativity of $\varphi_1$ and $\varphi_2$, $\mu_1$ and $\mu_2$ can always be chosen such that $\mu_1(S_1)=\varphi_1(S_1)$ and $\mu_2(S_2)=\varphi_2(S_2)$. Therefore, the next theorem settles the existence of a coupling of two nonnegative submodular functions whenever $\varphi_1(\emptyset)=\varphi_2(\emptyset)=0$.

\begin{thm} \label{thm:submod}
Let $\varphi_1\colon 2^{S_1} \to \bR_{\ge 0}$ and $\varphi_2 \colon 2^{S_2} \to \bR_{\ge 0}$ be submodular functions over ground sets $S_1$ and $S_2$, respectively, satisfying $\varphi_1(\emptyset)=0$ and $\varphi_2(\emptyset)=0$. Furthermore, let $\mu_1\colon S_1 \to \bR_{\ge 0}$ and $\mu_2\colon S_2 \to \bR_{\ge 0}$ be such that $\mu_1(S_1)=\varphi_1(S_1)$ and $\mu_2(S_2)=\varphi_2(S_2)$. Then, the function $b$ defined in~\eqref{eq:submod_coupling} is a submodular coupling of $\varphi_1$ and $\varphi_2$.
\end{thm}
\begin{proof}
Take arbitrary $U,V\subseteq S_1\times S_2$, $e_1\in S_1$, and $e_2\in S_2$. Then we have 
\begin{align*}
    \pi_2((U\cap V)_{e_1})&=\pi_2(U_{e_1})\cap \pi_2(V_{e_1}), \\
    \pi_2((U\cup V)_{e_1})&=\pi_2(U_{e_1})\cup \pi_2(V_{e_1}), \\
    \pi_1((U\cap V)^{e_2})&=\pi_1(U^{e_2})\cap \pi_1(V^{e_2}), \\
    \pi_1((U\cup V)^{e_2})&=\pi_1(U^{e_2})\cup \pi_1(V^{e_2}). 
\end{align*}
Together with the submodularity of $\varphi_1$ and $\varphi_2$ and the nonnegativity of $\mu_1$ and $\mu_2$, this implies that $b$ is a nonnegative combination of submodular functions minus a modular function, and is therefore submodular.

It remains to show that $b$ is a coupling. For each $Y_1\subseteq S_1$ and $Y_2\subseteq S_2$, we get
\begin{align} 
    b(Y_1\times Y_2)
    &=\mu_1(Y_1)\cdot\varphi_2(Y_2)+\mu_2(Y_2)\cdot\varphi_1(Y_1)-\mu_1(Y_1)\cdot\mu_2(Y_2) \label{eq:Y1Y2_first}\\
    &=\varphi_1(Y_1)\cdot\varphi_2(Y_2)- (\varphi_1(Y_1)-\mu_1(Y_1))\cdot (\varphi_2(Y_2)-\mu_2(Y_2)). \label{eq:Y1Y2}
\end{align}
Thus, by $\mu_1(S_1)=\varphi_1(S_1)$ and $\mu_2(S_2)=\varphi_2(S_2)$, we have $b(Y_1\times Y_2)=\varphi_1(Y_1)\cdot \varphi_2(Y_2)$ whenever $Y_1=S_1$ or $Y_2=S_2$, concluding the proof of the theorem.
\end{proof}

It is worth emphasizing that if $\varphi_1$ and $\varphi_2$ are integer-valued, they admit an integer-valued submodular coupling. This follows by choosing $\mu_1$ and $\mu_2$ to be integer-valued in the proof.

\begin{rem}
    It is natural to ask whether every submodular coupling can be obtained from the construction in \eqref{eq:submod_coupling} for a suitable choice of $\mu_1$ and $\mu_2$. The answer is no, as the following example shows. Let $M_1=(S_1,r_1)$ and $M_2=(S_2,r_2)$ be matroids whose rank functions are not modular, and suppose that they admit a tensor product $M=(S_1\times S_2,r)$. We claim that $r$ cannot be obtained in the form of \eqref{eq:submod_coupling}.
    Suppose, for contradiction, that there exist functions $\mu_1\colon S_1\to \bR_{\ge 0}$ and $\mu_2\colon S_2\to \bR_{\ge 0}$ such that $r$ is given by~\eqref{eq:submod_coupling} with $\varphi_i=r_i$.
    Applying \eqref{eq:Y1Y2} to arbitrary sets $Y_1\subseteq S_1$ and $Y_2\subseteq S_2$, and using $r(Y_1\times Y_2)=r_1(Y_1)\cdot r_2(Y_2)$, 
    we obtain that either $r_1(Y_1)=\mu_1(Y_1)$ or $r_2(Y_2)=\mu_2(Y_2)$. Since $Y_1$ and $Y_2$ were arbitrary, this implies that, for at least one of $i=1,2$, we have $\mu_i(Y_i)=r_i(Y_i)$ for every $Y_i\subseteq S_i$. But then $r_i$ is modular, contradicting our assumption.
\end{rem}

\subsection{Polymatroid functions}
\label{sec:polymatroid}

A noteworthy feature of the proof of Theorem~\ref{thm:submod} is that it relies on two arbitrarily chosen nonnegative modular functions satisfying $\mu_1(S_1)=\varphi_1(S_1)$ and $\mu_2(S_2)=\varphi_2(S_2)$. However, the resulting function $b$ may not be increasing, even if $\varphi_1$ and $\varphi_2$ are so. Therefore, to show that polymatroid functions admit a polymatroid coupling, $\mu_1$ and $\mu_2$ must be chosen carefully.

\begin{thm} \label{thm:poly}
    Let $\varphi_1\colon 2^{S_1} \to \bR_{\ge 0}$ be a $k_1$-polymatroid function and $\varphi_2\colon 2^{S_2} \to \bR_{\ge 0}$ be a $k_2$-polymatroid function over ground sets $S_1$ and $S_2$, respectively. Then, $\varphi_1$ and $\varphi_2$ have a $(k_1\cdot k_2)$-polymatroid coupling.
    Moreover, if $\varphi_1$ and $\varphi_2$ are integer-valued, then this coupling can be chosen to be integer-valued.
\end{thm}
\begin{proof}
Let $\mu_1 \in B(\varphi_1)$ and $\mu_2 \in B(\varphi_2)$ be arbitrary elements of the base polyhedra of $\varphi_1$ and $\varphi_2$, respectively, which are chosen to be integer-valued if $\varphi_1$ and $\varphi_2$ are integer-valued. Note that such elements exist by Proposition~\ref{prop:empty}. Let $b$ denote the submodular coupling of $\varphi_1$ and $\varphi_2$ defined as in \eqref{eq:submod_coupling} using $\mu_1$ and $\mu_2$.

\begin{cl} \label{cl:b} 
We have $b(\emptyset) = 0$ and $b(Y_1 \times Y_2) \le \varphi_1(Y_1)\cdot  \varphi_2(Y_2)$  for each $Y_1 \subseteq S_1$ and $Y_2 \subseteq S_2$, where equality holds if $\varphi_1(Y_1) =\mu_1(Y_1)$ or $\varphi_2(Y_2)=\mu_2(Y_2)$. 
\end{cl}
\begin{proof}
The equality $b(\emptyset) = 0$ follows from $\varphi_1(\emptyset)=\varphi_2(\emptyset) = 0$.
For each $Y_1 \subseteq S_1$ and $Y_2 \subseteq S_2$, using \eqref{eq:Y1Y2}, $\mu_1(Y_1) \le \varphi_1(Y_1)$ and $\mu_2(Y_2) \le \varphi_2(Y_2)$, we get
\[
b(Y_1 \times Y_2) =
\varphi_1(Y_1)\cdot\varphi_2(Y_2)- (\varphi_1(Y_1)-\mu_1(Y_1))\cdot (\varphi_2(Y_2)-\mu_2(Y_2)) \le 
\varphi_1(Y_1)\cdot  \varphi_2(Y_2). \]
Here equality holds if $\varphi_1(Y_1) =\mu_1(Y_1)$ or $\varphi_2(Y_2)=\mu_2(Y_2)$.
\end{proof}

The function $b$ thus obtained is still not necessarily increasing. To get an increasing set function, we define $\varphi$ on $S$ by setting \[\varphi(Z) \coloneqq \min\{b(Z') \mid Z' \supseteq Z\}.\]
The following technical claim is folklore; see e.g.~\cite{frank2011connections}. We include a proof to make the paper self-contained.

\begin{cl}\label{cl:inc}
$\varphi$ is an increasing submodular function over $S_1\times S_2$.
\end{cl}
\begin{proof}
Clearly, $\varphi$ is increasing by definition. To prove submodularity, let $U, V\subseteq S_1\times S_2$. Then, there exist $U'\supseteq U$ and $V'\supseteq V$ such that $\varphi(U)=b(U')$ and $\varphi(V)=b(V')$. Then, using the submodularity of $b$ and the definition of $\varphi$, we get 
\begin{align*}
\varphi(U)+\varphi(V)
&=
b(U')+b(V')\\
&\ge 
b(U'\cap V')+b(U'\cup V')\\
&\ge
\varphi(U\cap V)+\varphi(U\cup V),
\end{align*}
proving the claim.
\end{proof}

Finally, we show that $\varphi$ and $b$ coincide on product sets.

\begin{cl} \label{cl:varphi_on_product}
$\varphi(Y_1\times Y_2) = b(Y_1\times Y_2)$ holds for each $Y_1\subseteq S_1$ and $Y_2 \subseteq S_2$.
\end{cl}
\begin{proof}
We need to show that $b(Z) \ge b(Y_1\times Y_2)$ holds if $Y_1\times Y_2 \subseteq Z \subseteq S_1 \times S_2$. Using that $\varphi_2$ is increasing and $\mu_2(X) \le \varphi_2(X)$ holds for any $X\subseteq S_2$, we get
\begin{align}
\sum_{e_1\in S_1}\mu_1(e_1)\cdot \varphi_2(\pi_2(Z_{e_1})) & \ge \sum_{e_1 \in Y_1} \mu_1(e_1) \cdot \varphi_2(Y_2) + \sum_{e_1 \in S_1\setminus Y_1} \mu_1(e_1) \cdot \varphi_2(\pi_2(Z_{e_1})) \nonumber \\
& \ge \sum_{e_1 \in Y_1} \mu_1(e_1) \cdot \varphi_2(Y_2) + \sum_{e_1 \in S_1\setminus Y_1} \mu_1(e_1) \cdot \mu_2(\pi_2(Z_{e_1})) \nonumber \\
&= \mu_1(Y_1)\cdot  \varphi_2(Y_2) + \sum_{\substack{(e_1, e_2) \in Z \setminus (Y_1 \times S_2)}}  \mu_1(e_1) \cdot\mu_2(e_2). \label{eq:ineq1}
\end{align}
Similarly, using that $\varphi_1$ is increasing and $\mu_1(X) \le \varphi_1(X)$ holds for any $X\subseteq S_1$, we get
\begin{align}
\sum_{e_2\in S_2}\mu_2(e_2)\cdot \varphi_1(\pi_1(Z^{e_2})) & \ge \sum_{e_2 \in Y_2} \mu_2(e_2) \cdot \varphi_1(Y_1) + \sum_{e_2 \in S_2\setminus Y_2} \mu_2(e_2) \cdot \varphi_1(\pi_1(Z^{e_2})) \nonumber \\
& \ge \sum_{e_2 \in Y_2} \mu_2(e_2) \cdot \varphi_1(Y_1) + \sum_{e_2 \in S_2\setminus Y_2} \mu_2(e_2) \cdot \mu_1(\pi_1(Z^{e_2})) \nonumber \\
&= \mu_2(Y_2)\cdot  \varphi_1(Y_1) + \sum_{\substack{(e_1, e_2) \in Z \setminus (S_1 \times Y_2)}}  \mu_1(e_1) \cdot\mu_2(e_2). \label{eq:ineq2}
\end{align}
Using the definition \eqref{eq:submod_coupling} of $b$, the equality \eqref{eq:Y1Y2_first}, and the sum of the inequalities \eqref{eq:ineq1} and \eqref{eq:ineq2}, we get
\begin{align*}
b(Z) &\ge  \mu_1(Y_1)\cdot \varphi_2(Y_2) + \sum_{\substack{(e_1, e_2) \in Z \setminus (Y_1 \times S_2)}} \mu_1(e_1)\cdot \mu_2(e_2)  + \mu_2(Y_2)\cdot \varphi_1(Y_1) + \sum_{\substack{(e_1, e_2) \in Z \setminus (S_1 \times Y_2)}}  \mu_1(e_1)\cdot \mu_2(e_2)\\
&~~~~- \sum_{(e_1,e_2)\in Z} \mu_1(e_1) \cdot \mu_2(e_2) \\
& = \mu_1(Y_1)\cdot \varphi_2(Y_2) + \mu_2(Y_2)\cdot \varphi_1(Y_1) + \sum_{\substack{(e_1, e_2) \in Z\\ e_1 \not \in Y_1, e_2 \not \in Y_2}} \mu_1(e_1)\cdot \mu_2(e_2) - \sum_{(e_1, e_2) \in Y_1 \times Y_2} \mu_1(e_1) \cdot\mu_2(e_2) \\
& \ge \mu_1(Y_1)\cdot \varphi_2(Y_2) + \mu_2(Y_2)\cdot \varphi_1(Y_1) - \mu_1(Y_1)\cdot\mu_2(Y_2) \\
& = b(Y_1 \times Y_2).
\end{align*}
This concludes the proof of the claim.
\end{proof}

By Claim~\ref{cl:inc}, $\varphi$ is increasing submodular. Claim~\ref{cl:varphi_on_product} implies that $\varphi(\emptyset) = b(\emptyset) = 0$, and thus $\varphi \ge 0$. Using $\varphi_1(S_1)=\mu_1(S_1)$, $\varphi_2(S_2)=\mu_2(S_2)$, and Claims~\ref{cl:b} and \ref{cl:varphi_on_product}, we get that $\varphi(Y_1 \times S_2) = b(Y_1\times S_2) = \varphi_1(Y_1) \cdot \varphi_2(S_2)$ and $\varphi(S_1 \times Y_2) = b(S_1 \times Y_2) = \varphi_1(S_1) \cdot \varphi_2(Y_2)$ hold for $Y_1 \subseteq S_1$ and $Y_2 \subseteq S_2$, that is, $\varphi$ is a coupling of $\varphi_1$ and $\varphi_2$.
For $(e_1, e_2)\in S_1\times S_2$, by Claims~\ref{cl:b} and \ref{cl:varphi_on_product}, we have
\[\varphi((e_1, e_2)) = b((e_1, e_2)) \le \varphi_1(e_1)\cdot \varphi_2(e_2) \le k_1 \cdot k_2,\]
that is, $\varphi$ is a $(k_1\cdot k_2)$-polymatroid function. Finally, if $\varphi_1$ and $\varphi_2$ are integer-valued, then $\varphi$ is integer-valued as well by construction. This finishes the proof of the theorem.
\end{proof}

\subsection{Matroids and tensor products}
\label{sec:matroids}

One of the main motivations for our work was to extend the notion of coupling to matroids. For ease of discussion, we say that $M=(S_1\times S_2,r)$ is a coupling of matroids $M_1=(S_1,r_1)$ and $M_2=(S_2,r_2)$ if $r$ is a coupling of $r_1$ and $r_2$. In Section~\ref{sec:local}, we show that matroids always admit a coupling that nearly satisfies the requirements of being a tensor product. Our construction, as it turns out, has deep connections to amalgams of matroids. Then, in Section~\ref{sec:tensor}, we characterize couplings that are tensor products as well. 
Finally, in Section~\ref{sec:ingleton}, we show that the existence of a tensor product with $U_{2,3}$ implies Ingleton's inequality, providing a necessary condition for matroid representability.

\subsubsection{Matroid couplings and amalgams}
\label{sec:local}

Since the rank function of any matroid is an integer-valued $1$-polymatroid function, Theorem~\ref{thm:poly} can be applied. Actually, the proof implies an even stronger result.

\begin{cor} \label{cor:matroid}
Let $M_1=(S_1, r_1)$ and $M_2=(S_2, r_2)$ be matroids, let $B_1$ be a basis of $M_1$ and $B_2$ be a basis of $M_2$, and set $S\coloneqq S_1\times S_2$. Then, $M_1$ and $M_2$ have a coupling $M=(S,r)$ such that $r(Y_1\times Y_2)=r_1(Y_1)\cdot r_2(Y_2)$ whenever $Y_1\subseteq B_1$ or $Y_2\subseteq B_2$. 
\end{cor}
\begin{proof}
Let $\mu_1\coloneqq \chi_{B_1}$ and $\mu_2\coloneqq\chi_{B_2}$ be the characteristic vectors of $B_1$ and $B_2$, respectively. Clearly, the inclusion $\mu_i\in B(r_i)$ holds for $i=1,2$. By Theorem~\ref{thm:poly}, $M_1$ and $M_2$ admit a coupling $M=(S,r)$, where 
\begin{equation*}
    r(Z)=\min_{W\supseteq Z}\left\{\sum_{e_1\in B_1}r_2(\pi_2(W_{e_1}))+\sum_{e_2\in B_2}r_1(\pi_1(W^{e_2})) - |W\cap(B_1\times B_2)|\right\}.
\end{equation*}
Let $Y_1\subseteq S_1$ and $Y_2\subseteq S_2$. By Claim~\ref{cl:varphi_on_product}, the minimum is attained at $W=Z$ for $Z=Y_1\times Y_2$. Furthermore, if $Y_1\subseteq B_1$ or $Y_2\subseteq  B_2$, then $r_1(Y_1)=\mu_1(Y_1)$ or $r_2(Y_2)=\mu_2(Y_2)$, and thus $r(Y_1\times Y_2)=r_1(Y_1)\cdot r_2(Y_2)$ by Claim~\ref{cl:b}.
This completes the proof of the corollary.
\end{proof}

\begin{rem}
    To illustrate that different choices of bases can lead to non-isomorphic couplings, let $M_1=U_{1,2}$ on $S_1=\{a,b\}$, and let $M_2$ be the loopless rank-$2$ matroid on $S_2=\{x_1,x_2,y,z\}$ in which $x_1$ and $x_2$ are parallel and there are no other parallel pairs. Fix $B_1=\{a\}$. In the construction of \cref{cor:matroid}, the restriction of the coupling to $\{a\}\times S_2$ is isomorphic to $M_2$. Moreover, each $(b,e)$ with $e\in B_2$ is parallel to $(a,e)$, while the elements $(b,e)$ with $e\notin B_2$ are loops. If we choose $B_2=\{x_1,y\}$, the non-loop parallel classes of the resulting coupling have sizes $3$, $2$ and $1$. On the other hand, if we choose $B_2=\{y,z\}$, they all have size $2$. Thus, the two couplings are not isomorphic.
\end{rem}

\begin{rem}
    For $r,n\in \bZ_{\ge 0}$, let $N=N_{r,n}$ denote the number of rank-$r$ matroids on $n$ elements up to isomorphism. There is an elementary way to combine all these matroids into a single matroid using direct sums. Namely, choose representatives $M_1,\ldots,M_N$ of the isomorphism classes of rank-$r$ matroids on $[n]$, and take their direct sum. Then each $M_i$ appears as the restriction to the corresponding component, and hence every rank-$r$ matroid on $n$ elements occurs, up to isomorphism, as a minor.

    As a consequence of Corollary~\ref{cor:matroid}, we now describe a different construction with a stronger product-type structure. By iteratively applying the coupling construction from Corollary~\ref{cor:matroid} to the matroids $M_1,\ldots,M_N$, we obtain a matroid $M$ on the ground set $[n]^N$ of rank $r^N$ with the following property. For every $i\in [N]$, the projection of the rank function of $M$ onto the $i$th coordinate recovers the rank function of $M_i$ up to the constant factor $r^{N-1}$.
\end{rem}

Recall that two matroids do not necessarily admit a tensor product. However, the main message of Corollary~\ref{cor:matroid} is that there always exists a coupling that locally satisfies the properties of the tensor product in the sense that $r(Y_1 \times Y_2) = r_1(Y_1) \cdot r_2(Y_2)$ holds, not universally for all $Y_1, Y_2$, but specifically when at least one of $Y_1 \subseteq B_1$ and $Y_2 \subseteq B_2$ is satisfied. In particular, this implies that $M|(S_1 \times \{e_2\})$ is isomorphic to $M_1$ for each $e_2 \in B_2$ by the natural bijection to $S_1$, and $M|(\{e_1\} \times S_2)$ is isomorphic to $M_2$ for each $e_1 \in B_1$ by the natural bijection to $S_2$.

\begin{rem}
In the previous discussion, our strategy was to verify the existence of a coupling first for polymatroid functions and then for matroids. However, one could also take the reverse approach. The construction in Theorem~\ref{thm:poly} can be formalized directly for matroids, as demonstrated in the proof of Corollary~\ref{cor:matroid}. Since every integer-valued polymatroid function is the quotient of a matroid rank function by Proposition~\ref{prop:helgason}, this immediately implies the existence of a coupling for integer-valued polymatroid functions. Extending this, one can readily establish the statement for rational-valued polymatroid functions. Finally, by approximating a real-valued polymatroid function with rational-valued ones, one obtains the existence of a coupling for general polymatroid functions.    
\end{rem}

Let $N_0 = (T_1\cap T_2, r_0')$, $N_1 = (T_1, r'_1)$, and $N_2=(T_2, r'_2)$ be matroids such that $N_1|(T_1 \cap T_2) = N_2|(T_1 \cap T_2) = N_0$, and let $T\coloneqq T_1 \cup T_2$. A matroid $N=(T,r)$ is called an {\it amalgam} of $N_1$ and $N_2$ if $N|T_1 = N_1$ and $N|T_2 = N_2$. An amalgam $N$ is called {\it free} if for any amalgam $N'$, each independent set of $N'$ is independent in $N$. The following result appears, for example, in~\cite[Proposition~11.4.2]{oxley2011matroid}.

\begin{prop}\label{prop:amalgam}
Let $N_0 = (T_1 \cap T_2, r')$, $N_1 = (T_1, r'_1)$, and $N_2=(T_2, r'_2)$ be matroids such that $N_1|(T_1 \cap T_2) = N_2|(T_1 \cap T_2) = N_0$, and let $T\coloneqq T_1 \cup T_2$. Define $r\colon 2^T\to\bZ_{\ge 0}$ by
\[r(Z) \coloneqq \min \{r'_1(Y\cap T_1) + r'_2(Y \cap T_2) - r'(Y \cap T_1 \cap T_2) \mid Y \supseteq Z\}.
\]
If $N_0$ is a modular matroid, then $r$ is submodular. Moreover, if $r$ is submodular, then $N = (T, r)$ is a free amalgam of $N_1$ and $N_2$.
\end{prop}

If the function $r$ defined in Proposition~\ref{prop:amalgam} is submodular, then $N=(T, r)$ is called the {\it proper amalgam} of $N_1$ and $N_2$. Both free and proper amalgams are unique if they exist. It turns out, though, that for a given pair of matroids, there may be no free amalgam, or the free amalgam may not be proper, or no amalgams may exist at all. 

Given matroids $M_1=(S_1, r_1)$ and $M_2=(S_2, r_2)$, a basis $B_1$ of $M_1$, and a basis $B_2$ of $M_2$, consider the matroids $N_1=\bigoplus_{e_2 \in B_2} M_1$ on ground set $S_1 \times B_2$ and $N_2 = \bigoplus_{e_1 \in B_1} M_2$ on ground set $B_1 \times S_2$. Observe that $(S_1 \times B_2) \cap (B_1 \times S_2) = B_1 \times B_2$ and $N_1|(B_1 \times B_2) = N_2|(B_1 \times B_2)$ is the free matroid on $B_1 \times B_2$. Therefore, Proposition~\ref{prop:amalgam} implies that the proper amalgam of $N_1$ and $N_2$ exists. Our construction for the coupling of $M_1$ and $M_2$ shown in the proof of \cref{cor:matroid} can be obtained from the proper amalgam of $N_1$ and $N_2$ by adding the elements of $(S_1 \setminus B_1) \times (S_2 \setminus B_2)$ as loops. We prove that the same construction gives a coupling for any amalgam of $N_1$ and $N_2$.

\begin{thm} \label{thm:coupalgam}
Let $M_1=(S_1, r_1)$ and $M_2=(S_2, r_2)$ be matroids, and let $B_1$ be a basis of $M_1$ and $B_2$ be a basis of $M_2$. Consider the matroids $N_1=\bigoplus_{e_2 \in B_2} M_1$ on ground set $S_1 \times B_2$ and $N_2 = \bigoplus_{e_1 \in B_1} M_2$ on ground set $B_1 \times S_2$, and let $S\coloneqq S_1 \times S_2$. If $M=(S, r)$ is a matroid such that each element of $(S_1 \setminus B_1) \times (S_2 \setminus B_2)$ is a loop in $M$ and $M|((S_1 \times B_2) \cup (B_1 \times S_2))$ is an amalgam of $N_1$ and $N_2$, then $M$ is a coupling of $M_1$ and $M_2$.  
\end{thm}
\begin{proof}
We show that $r(Y_1 \times S_2) = r_1(Y_1)\cdot r_2(S_2)$ holds for each $Y_1 \subseteq S_1$; a similar argument proves that $r(S_1 \times Y_2) = r_1(S_1)\cdot r_2(Y_2)$ holds for each $Y_2 \subseteq S_2$. Since the set $B_1\cap Y_1$ is independent in $M_1$, it can be extended to a maximum independent set $I_1$ of $Y_1$ with $|I_1| = r_1(Y_1)$. Since the set $I_1 \times B_2$ is independent in $N_1$, it is independent in $M$; thus $r(Y_1 \times S_2) \ge |I_1 \times B_2| = r_1(Y_1)\cdot r_2(S_2)$. To see the reverse inequality, observe that $I_1 \times B_2$ spans $Y_1 \times B_2$ in $N_1$,  $(B_1 \cap Y_1) \times B_2 \subseteq I_1 \times B_2$ spans $(B_1 \cap Y_1) \times S_2$ in $N_2$, and each element of $(Y_1 \setminus B_1) \times (S_2 \setminus B_2)$ is a loop of $M$. Thus $I_1 \times B_2$ spans $Y_1 \times S_2$ in $M$, implying $r(Y_1 \times S_2) \le |I_1 \times B_2| = r_1(Y_1)\cdot r_2(S_2)$.
\end{proof}

\subsubsection{Tensor products}
\label{sec:tensor}

From the definitions alone, it is clear that coupling and the tensor product are closely related concepts. The following theorem characterizes which couplings satisfy the stronger properties of the tensor product.

\begin{thm} \label{thm:tensor}
Let $M_1=(S_1, r_1)$, $M_2=(S_2, r_2)$, and $M=(S, r)$ be matroids such that $S=S_1 \times S_2$. Then, the following are equivalent:
\begin{enumerate}[label=(\roman*)]\itemsep0em
\item $r(Y_1\times Y_2) = r_1(Y_1)\cdot r_2(Y_2)$ for each $Y_1 \subseteq S_1$ and $Y_2 \subseteq S_2$, \label{it:tensor}
\item $r(\{e_1\} \times Y_2) = r_1(e_1)\cdot r_2(Y_2)$ for each $Y_2\subseteq S_2$ and $e_1 \in S_1$, $r(Y_1\times \{e_2\}) = r_1(Y_1)\cdot r_2(e_2)$ for each $Y_1\subseteq S_1$ and $e_2 \in S_2$, and $r(S) = r_1(S_1)\cdot r_2(S_2)$,\label{it:fibers}
\item $M$ is a coupling of $M_1$ and $M_2$ and $r((e_1, e_2)) = r_1(e_1)\cdot r_2(e_2)$ for each $(e_1, e_2) \in S$. \label{it:coupling}
\end{enumerate}
\end{thm}
\begin{proof}
The equivalence of \ref{it:tensor} and \ref{it:fibers} was shown by Las Vergnas~\cite{las1981products}, and it is also clear that \ref{it:tensor} implies \ref{it:coupling}. Assume now that \ref{it:coupling} holds; we prove that it implies \ref{it:fibers}. For the proof, we will use the following simple observation.

\begin{cl} \label{cl:direct_sum}
If $I_1\subseteq S_1$ is independent in $M_1$ and $I_2\subseteq S_2$ is independent in $M_2$, then 
\[M|(I_1 \times S_2) = \bigoplus\limits_{e_1 \in I_1} M|(\{e_1\} \times S_2) \text{ and } M|(S_1 \times I_2) = \bigoplus\limits_{e_2 \in I_2} M|(S_1 \times \{e_2\}).\]
\end{cl}
\begin{proof}
Using that $M$ is a coupling, we get that \[\sum_{e_1 \in I_1} r(\{e_1\} \times S_2) = \sum_{e_1 \in I_1} r_1(e_1) \cdot r_2(S_2) = |I_1| \cdot r_2(S_2) = r_1(I_1)\cdot  r_2(S_2) = r(I_1 \times S_2),\] implying 
 $M|(I_1 \times S_2) = \bigoplus_{e_1 \in I_1} M|(\{e_1\} \times S_2)$. The second statement follows similarly.
\end{proof}

First we show that $r(Y_1\times Y_2) \ge r_1(Y_1)\cdot r_2(Y_2)$ holds for each $Y_1 \subseteq S_1$ and $Y_2 \subseteq S_2$. Let $I_1 \subseteq Y_1$ and $I_2 \subseteq Y_2$ be maximal independent sets of $M_1$ and $M_2$ contained in $Y_1$ and $Y_2$, respectively.
\cref{cl:direct_sum} implies that each circuit contained in $I_1 \times S_2$ intersects at most one of the sets of the form $\{e_1\} \times S_2$ for $e_1 \in I_1$. Similarly, each circuit contained in $S_1 \times I_2$ intersects at most one of the sets of the form $S_1 \times \{e_2\}$ for $e_2 \in I_2$. These show that each circuit in $I_1 \times I_2$ is a loop. Since the equality $r((e_1, e_2)) = r_1(e_1)\cdot r_2(e_2) = 1$ holds for $(e_1, e_2) \in I_1 \times I_2$ by assumption, the set $I_1\times I_2$ does not contain any loops. This proves that $I_1 \times I_2$ is independent in $M$; hence, $r(Y_1 \times Y_2) \ge r_1(Y_1)\cdot r_2(Y_2)$.

Next we show that $r(\{e_1\} \times Y_2) = r_1(e_1)\cdot r_2(Y_2)$ holds for each $e_1 \in S_1$ and $Y_2 \subseteq S_2$. If $e_1$ is a loop in $M_1$, then $r(\{e_1\} \times S_2) = r_1(e_1)\cdot r_2(S_2) = 0$, and thus $r(\{e_1\} \times Y_2) = 0$. Otherwise, let $B_1$ be a basis of $M_1$ containing $e_1$. By the previous paragraph, we have $r(\{b_1\} \times Y_2) \ge r_2(Y_2)$ for $b_1 \in B_1$. Furthermore, $M|(B_1 \times Y_2) = \bigoplus_{b_1 \in B_1} M|(\{b_1\} \times Y_2)$ by \cref{cl:direct_sum}. Thus,
\[r(S_1\times Y_2) \ge r(B_1 \times Y_2) = \sum_{b_1 \in B_1} r(\{b_1\} \times Y_2) \ge \sum_{b_1 \in B_1} r_2(Y_2) = r_1(S_1)\cdot r_2(Y_2) = r(S_1\times Y_2),\]
where the last equality holds by the assumption that $M$ is a coupling. Therefore, equality holds throughout, implying that $r(\{b_1\} \times Y_2) = r_2(Y_2)$ for each $b_1 \in B_1$. In particular, this shows that $r(\{e_1\} \times Y_2) = r_1(e_1)\cdot r_2(Y_2)$. Similarly, we get that $r(Y_1 \times \{e_2\}) = r_1(Y_1) \cdot r_2(e_2)$ for each $Y_1 \subseteq S_1$ and $e_2 \in S_2$. Finally, $r(S)=r_1(S_1)\cdot r_2(S_2)$ follows from the assumption that $M$ is a coupling.
\end{proof}

If $e_1$ or $e_2$ is a loop in the corresponding matroid, then $(e_1,e_2)$ is a loop in any coupling by definition. Hence if $(e_1,e_2)$ is not a loop in a coupling, then $r((e_1,e_2))=r_1(e_1)\cdot r_2(e_2)$ holds. Therefore, an interesting consequence of Theorem~\ref{thm:tensor} is the following.

\begin{cor}
Any loopless coupling of two matroids is also a tensor product.
\end{cor}

\subsubsection{Linear matroids and Ingleton's inequality}
\label{sec:ingleton}

Tools for determining the linear representability of a matroid include the use of excluded minors and the inequality of Ingleton~\cite{ingleton1971representation}, who showed that the rank function of any linear matroid $M=(S,r)$ satisfies 
\begin{align}
r(A\cup B) + r(A\cup C)+r(A\cup D)+r(B\cup C)+r(B\cup D) \ge \nonumber \\
r(A)+r(B)+r(A\cup B \cup C) + r(A \cup B \cup D)+r(C\cup D) 
\label{ineq:ingleton}
\end{align}
for every $A,B,C,D\subseteq S$. In this sense, the inequality describes a necessary condition for linearity. Since the uniform matroid $U_{2,3}$ is regular, it has a matroid tensor product with any linear matroid. In other words, for a matroid, the existence of a matroid tensor product with $U_{2,3}$ is also a necessary condition for linearity. Our next result shows that this property not only implies Ingleton's inequality for polymatroid functions but also yields a slightly weaker form for general submodular functions.

\begin{thm} \label{thm:tensor_ingleton}
Let $\varphi_1\colon 2^{S_1}\to \bR$ be a submodular function and let $\varphi_2\colon 2^{S_2} \to \bR_{\ge 0}$ denote the rank function of the uniform matroid $U_{2,3}$. 
\begin{enumerate}[label=(\alph*)]\itemsep0em
\item  If $\varphi_1$ and $\varphi_2$ have a submodular tensor product, then $\varphi_1$ satisfies Ingleton's inequality \eqref{ineq:ingleton} for all pairwise disjoint sets $A, B, C, D \subseteq S_1$. \label{it:ingleton1}
\item If $\varphi_1$ and $\varphi_2$ have a tensor product which is a polymatroid function, then  $\varphi_1$ satisfies Ingleton's inequality \eqref{ineq:ingleton} for all sets $A, B, C, D \subseteq S_1$. \label{it:ingleton2}
\end{enumerate}
\end{thm}

\begin{proof}
Let $A, B, C, D \subseteq S_1$, $S\coloneqq S_1 \times S_2$, and $\varphi \colon 2^S \to \bR$ be a submodular tensor product of $\varphi_1$ and $\varphi_2$ that, in the case of \ref{it:ingleton2}, is also increasing. We denote the elements of $S_2$ by $1$, $2$, and $3$. To improve readability, we use several abbreviations in the following inequalities: for $H=\{h_1,\dots, h_k\}\subseteq \{1,2,3\}$, we use $h_1\dots h_k$ to denote $H$; for $X\in \{A, B, C, D\}$, we let $X_H$ denote $X \times H$. Lastly, we omit the $\cup$ symbols; e.g.,  $AB\times 12$ denotes $(A\cup B) \times \{1,2\}$ and $A_{12}B_{13}$ denotes $(A \times \{1,2\}) \cup (B\times \{1,3\})$. Consider the following inequalities.
\begin{align*}
\varphi(AB \times 2) + \varphi(A_{123} C_1) & \ge \varphi(A\times 2) + \varphi(A_{123} B_2 C_1) \\ 
\varphi(AC \times 1 ) + \varphi(A \times 123) & \ge \varphi(A\times 1) + \varphi(A_{123}  C_1) \\ 
\varphi(AD \times 3) + \varphi(A_3 B_2 C_1) & \ge \varphi(A \times 3) + \varphi(A_3 B_2 C_1 D_3) \\ 
\varphi(BD \times 2) + \varphi(A_3 B_2  C_1  D_3) & \ge \varphi(B \times 2) + \varphi(A_3 B_2  C_1  D_{23}) \\ 
\varphi(ABC \times 3) + \varphi(BC \times 123)  & \ge \varphi(BC \times 3) + \varphi(A_3 B_{123} C_{123}) \\
\varphi(ABCD \times 1) + \varphi(A_3 B_2 C_1 D_{123}) & \ge \varphi(CD \times 1) + \varphi(A_{13}  B_{12}  C_1 D_{123}) \\
\varphi(ABCD \times 2) + \varphi(A_{123}  B_{12}  C_1  D_{123}) & \ge \varphi(ABD \times 2) + \varphi(A_{123} B_{12} C_{12} D_{123}) \\
\varphi(A \times 123) + \varphi(A_{13} B_{12} C_1 D_{123}) & \ge \varphi(A \times 13) + \varphi(A_{123} B_{12}  C_1 D_{123}) \\
\varphi(A_{123}  B_2  C_1) + \varphi(A_3  B_{123}  C_{123}) & \ge \varphi(A_3 B_2 C_1) + \varphi(ABC \times 123) \\
\varphi(D \times 123) + \varphi(A_3 B_2 C_1 D_{23}) & \ge \varphi(D \times 23) + \varphi(A_3 B_2 C_1 D_{123}) \\
\varphi(BC \times 123) + \varphi(A_{123} B_{12} C_{12} D_{123}) & \ge \varphi(BC \times 12) + \varphi(ABCD \times 123)
\end{align*}
If $A$, $B$, $C$, and $D$ are pairwise disjoint, then each of the inequalities has the form $\varphi(U)+\varphi(V) \ge \varphi(U \cap V) + \varphi(U \cup V)$ for some $U, V \subseteq S$, and thus each of them holds by the submodularity of $\varphi$.  In general, if the sets $A$, $B$, $C$, and $D$ are not necessarily pairwise disjoint, then each inequality has the form $\varphi(U) + \varphi(V) \ge \varphi(Z) + \varphi(U \cup V)$ for some $U, V \subseteq S$ and $Z \subseteq U \cap V$, and thus holds if $\varphi$ is a polymatroid function. Therefore, under the assumptions of  \ref{it:ingleton1} or \ref{it:ingleton2}, each of these eleven inequalities holds.

In each of the inequalities, substitute $\varphi(Y_1 \times Y_2) = \varphi_1(Y_1)\cdot \varphi_2(Y_2) = \varphi_1(Y_1) \cdot \min \{|Y_2|, 2\}$ for each set of the form $Y_1 \times Y_2$ with $Y_1 \subseteq S_1$ and $Y_2 \subseteq S_2$. Then, from the sum of the above eleven inequalities, we obtain Ingleton's inequality \eqref{ineq:ingleton} for $\varphi_1$.
\end{proof}

It was shown by Las Vergnas~\cite{las1981products} that the Vámos matroid and the uniform matroid $U_{2,3}$ do not have a matroid tensor product. The Vámos matroid is the matroid on the ground set $\{a_1, a_2, b_1, b_2, c_1, c_2, d_1, d_2\}$ in which each subset of size 4 is a basis except for $A\cup B$, $A\cup C$, $A \cup D$, $B\cup C$, and $B\cup D$, where $A=\{a_1, a_2\}$, $B=\{b_1, b_2\}$, $C=\{c_1, c_2\}$, and $D=\{d_1, d_2\}$. As the rank function of the Vámos matroid does not satisfy Ingleton's inequality~\eqref{ineq:ingleton} for the pairwise disjoint sets $A$, $B$, $C$, and $D$, \cref{thm:tensor_ingleton} implies the following strengthening of the result of Las Vergnas~\cite{las1981products}.

\begin{cor}
The rank function of the Vámos matroid does not have a submodular tensor product with the rank function of the uniform matroid $U_{2,3}$.
\end{cor}

For a fixed ground set $S_1$ and a fixed matroid $N=(S_2,r_2)$, the submodular functions $\varphi_1\colon 2^{S_1}\to \bR$ that admit a tensor product with $r_2$ form a polyhedron in $\bR^{2^{S_1}}$. Indeed, this class is described by linear inequalities: those expressing submodularity, together with the additional linear constraints required for the existence of a tensor product with $r_2$. The same statement holds for polymatroid functions, with the corresponding polymatroid inequalities in place of submodularity. If we choose $N$ to be a regular matroid, then each of these inequalities is a {\it linear rank inequality}, that is, a linear inequality satisfied by the rank function of any representable matroid. 
Ingleton's inequality~\eqref{ineq:ingleton} was the first known linear rank inequality, and since then infinite families of independent linear rank inequalities have been found~\cite{dougherty2009linear, kinser2011new}. Interestingly, all currently known linear rank inequalities follow from the so-called {\it common information property} of the rank function of representable matroids~\cite{dougherty2009linear, bamiloshin2021common}, though other linear inequalities have been found that are satisfied by rank functions of matroids representable over fields of certain characteristics but not by all representable matroids~\cite{dougherty2014characteristic}.\footnote{After the submission of the current paper, the follow-up paper \cite{berczi2026interaction} identified a linear rank inequality that does not follow from the common information property.}
To define the common information property, we define $k$-CI-compliance recursively for $k \in \bZ_{\ge 0}$.
Let us call every polymatroid function {\it $0$-CI-compliant}, and call a polymatroid function $f\colon 2^S\to \bR$ {\it $k$-CI-compliant} if for any subsets $A, B \subseteq S$, there exist a ground set $S'\supseteq S$, a $(k-1)$-CI-compliant function $f'\colon 2^{S'}\to \bR$, and a subset $Z\subseteq S'$ such that
\begin{enumerate}[label=(\arabic*)] \itemsep0em
\item $f'(X) = f(X)$ for all $X \subseteq S$,
\item $f'(A\cup Z) = f(A)$ and $f'(B\cup Z) = f(B)$, 
\item $f'(Z) = f(A)+f(B)-f(A\cup B)$.
\end{enumerate}
Then, a polymatroid satisfies the {\it common information property} if it is {\it $k$-CI-compliant} for each integer $k \ge 0$.  Note that the common information property is not sufficient for the representability of a matroid: for example, all rank-3 matroids satisfy it~\cite{bamiloshin2021common}, while the non-Desargues matroid is not representable~\cite{ingleton1971representation}.
It would be interesting to see the relations between the class of polymatroid functions having a tensor product with $U_{2,3}$, the class of polymatroid functions having a tensor product with every regular matroid, and the class of polymatroid functions satisfying the common information property.

\subsection{Coverage and \texorpdfstring{$k$}{k}-alternating functions}
\label{sec:infinite}

Coverage functions naturally arise in applications like set cover, influence maximization, sensor placement, feature selection, and data summarization. The goal of this section is to show that if one of the functions is a coverage function, then there exist couplings with stronger properties. For some $k\in\bZ_{>0}$, let us call a function $\varphi\colon2^S\to\bR_{\ge 0}$ {\it $k$-alternating} if $\varphi(\emptyset)=0$ and it satisfies \ref{eq:kalt} for the given $k$.

\begin{thm} \label{thm:infk}
Let $\varphi_1\colon 2^{S_1} \to \bR_{\ge 0}$ be a coverage function and $\varphi_2\colon 2^{S_2}\to\bR_{\ge 0}$ be a $k$-alternating function over ground sets $S_1$ and $S_2$, respectively. Then, $\varphi_1$ and $\varphi_2$ have a $k$-alternating tensor product which is integer-valued if $\varphi_1$ and $\varphi_2$ are integer-valued. 
\end{thm}
\begin{proof} 
Assume first that $\varphi_1=\varphi_A$ for some $\emptyset\neq A\subseteq S_1$ as defined in \eqref{eq:extremal_def}. We define the tensor product of $\varphi_{A}$ and $\varphi_2$ to be 
\begin{equation*}
(\varphi_A\otimes \varphi_2)(X)\coloneqq \varphi_2(\pi_2((A\times S_2)\cap X))   
\end{equation*}
for any $X\subseteq S_1\times S_2$. For each pair of sets $Y_1\subseteq S_1$, $Y_2\subseteq S_2$, we have $(\varphi_A\otimes \varphi_2)(Y_1\times Y_2)=\varphi_2(Y_2)$ if $A\cap Y_1\neq \emptyset$ and $(\varphi_A\otimes \varphi_2)(Y_1\times Y_2)=0$ if $ A\cap Y_1=\emptyset$; thus $\varphi_A\otimes \varphi_2$ is indeed a tensor product. We claim that it is $k$-alternating as well. To see this, let $X_0,X_1, \dots, X_k\subseteq S_1\times S_2$. Then, the  $k$-alternating condition \ref{eq:kalt} requires that
\begin{align*}
\sum_{K \subseteq [k]} (-1)^{|K|}(\varphi_A\otimes \varphi_2) \bigl(X_0 \cup \bigcup_{i \in K} X_i\bigr)
&= 
\sum_{K \subseteq [k]} (-1)^{|K|}\varphi_2\Bigl(\pi_2{\bigl(}(A\times S_2)\cap(X_0 \cup \bigcup_{i \in K} X_i)\bigr)\Bigr)\\
&= 
\sum_{K \subseteq [k]} (-1)^{|K|}\varphi_2\Bigl(\pi_2{\bigl(}(A\times S_2)\cap X_0\bigr) \cup \bigcup_{i \in K} \pi_2 {\bigl(}(A\times S_2)\cap X_i)\bigr)\Bigr)\\
&=
\sum_{K \subseteq [k]} (-1)^{|K|}\varphi_2\bigl(Y_0 \cup \bigcup_{i \in K} Y_i\bigr)\\
&\le 
0.
\end{align*}
This follows from the $k$-alternating condition~\ref{eq:kalt} for $\varphi_2$ and sets $Y_i=\pi_2((A\times S_2)\cap X_i)$.

In general, if $\varphi_1$ is a coverage function, then by Proposition~\ref{prop:extremal}, $\varphi_1=\sum_{\emptyset\neq A\subseteq S_1} c_A\cdot \varphi_A$ for some nonnegative coefficients $c_A\in\bR_{\ge 0}$ for all nonempty $A\subseteq S_1$. Define the tensor product of $\varphi_1$ and $\varphi_2$ to be 
\begin{equation*}
    \varphi_1\otimes \varphi_2\coloneqq\sum_{\emptyset\neq A\subseteq S_1} c_A\cdot (\varphi_A \otimes \varphi_2).
\end{equation*}
Then, the resulting function over $S_1\times S_2$ is a tensor product of $\varphi_1$ and $\varphi_2$. Indeed, $(\varphi_1\otimes \varphi_2)(Y_1\times Y_2)= \sum_{\emptyset\neq A\subseteq S_1} c_A \cdot (\varphi_A \otimes \varphi_2) (Y_1\times Y_2)=(\sum_{A\cap Y_1\neq \emptyset} c_A)\cdot \varphi_2(Y_2)=\varphi_1(Y_1)\cdot \varphi_2(Y_2) $. Finally, observe that $\varphi_1\otimes\varphi_2$ is $k$-alternating as it is a nonnegative combination of $k$-alternating functions.
\end{proof}

As noted earlier, \ref{eq:kalt} for $k=1$ corresponds to the increasing property, while for $k=2$ it is equivalent to the increasing submodularity of the function. Thus, by applying Theorem~\ref{thm:infk} to the case when $k=2$, we get the following.

\begin{cor}\label{cor:inftens}
Any coverage function $\varphi_1$ and increasing submodular function $\varphi_2$ have an increasing submodular tensor product. 
\end{cor}

\begin{rem}
Note that if we do not assume $\varphi_2$ to be increasing, then a submodular tensor product might not exist. As an example, let $S_1 \coloneqq \{a, b\}$, $\varphi_1 \coloneqq \varphi_{S_1}$, $S_2 \coloneqq \{1, 2\}$, and $\varphi_2\colon 2^{S_2} \to \bR_{\ge 0}$ be defined by $\varphi_2(\emptyset)=\varphi_2(S_2)=0$ and $\varphi_2(\{1\})=\varphi_2(\{2\}) = 1$, and assume that $\varphi$ is a submodular tensor product of $\varphi_1$ and $\varphi_2$. Then, using the notation $a_i = (a,i)$ and $b_i = (b,i)$ for $i=1,2$, we have
\begin{align*}
\varphi(\{a_1, a_2, b_1\}) & = \varphi(\{a_1, a_2, b_1\}) + \varphi(\{a_1\}) - 1 \le \varphi(\{a_1, a_2\}) + \varphi(\{a_1, b_1\}) -1 = 0 + 1 -1 = 0, \\
\varphi(\{a_1, a_2, b_1\}) & = \varphi(\{a_1, a_2, b_1\}) + \varphi(\{b_1, b_2\}) \ge \varphi(\{b_1\}) + \varphi(\{a_1, a_2, b_1, b_2\}) = 1 + 0 = 1,
\end{align*}
a contradiction.
\end{rem}

Another interesting special case is when both $\varphi_1$ and $\varphi_2$ are coverage functions.

\begin{cor}\label{cor:inf}
Any two coverage functions have a coverage tensor product.
\end{cor}

\begin{rem}
In~\cite[Example 4.17]{bgils2026decomposition}, it is shown that the rank function of a matroid is a coverage function if and only if the matroid is a partition matroid. Combining this observation with Theorem~\ref{thm:infk} implies that any matroid has a tensor product with any partition matroid, which is a well-known result on tensor products.

The smallest non-partition matroid is $U_{2,3}$, whose rank function is therefore not a coverage function. Interestingly, any matroid that is not a partition matroid contains $U_{2,3}$ as a minor. Additionally, recalling that any two matroids representable over the same field admit a tensor product explains why the smallest example of a pair of matroids without a tensor product is the pair $U_{2,3}$ and the Vámos matroid.
\end{rem}

\subsection{Extension to the infinite case}
\label{sec:infground}

The earlier definitions extend naturally to set functions over an infinite domain. For clarity, we denote the ground set by $J$ rather than $S$ whenever it is infinite. In this context, let $(J, \cB)$ be a {\it set algebra}, i.e., $\cB$ is a family of subsets of $J$ that contains $\emptyset$ and is closed under taking complements and finite unions. Note that these imply that $\cB$ is also closed under taking finite intersections. If $\cB$ is further closed under taking countable unions, then $(J,\cB)$ is called a {\it measurable space}. A {\it set function} $\varphi$ assigns a real value to each member of $\cB$. By a slight abuse of notation, we call the members of $\cB$ {\it measurable}. For set algebras $(J_1,\cB_1)$ and $(J_2,\cB_2)$, a function $f\colon J_1\to J_2$ is {\it measurable} if $f^{-1}(B)\in\cB_1$ for all $B\in\cB_2$. For a set algebra $(J, \cB)$, we call a nonempty set $B \in \cB$ an {\it atom} of the algebra if $A \subseteq B$, $A \in \cB$ implies $A =\emptyset$ or $A=B$. We refer to a finitely additive measure, which is not necessarily nonnegative, as a {\it charge}. Properties of set functions, such as being normalized, increasing, decreasing, submodular, supermodular, modular, and being a polymatroid, can be generalized to functions over an infinite domain in a straightforward manner. Moreover, using the inequality-based definition, one can do the same for coverage functions. We call a set function {\it finite} if its domain is finite, and a set function is {\it bounded} if the range of the set function is bounded. 

Let $(J_1, \cB_1)$ and $(J_2, \cB_2)$ be measurable spaces, and let $J = J_1 \times J_2$. The {\it product algebra} $\cB$ on $J$ is defined as the algebra generated by $\{X_1 \times J_2 \mid X_1 \in \cB_1\} \cup \{J_1 \times X_2 \mid X_2 \in \cB_2\}$, and it is denoted by $\cB = \cB_1 \otimes \cB_2$; see e.g.~\cite{folland1999real}. We refer to the measurable space $(J, \cB)$ thus obtained as the {\it product} of measurable spaces $(J_1, \cB_1)$ and $(J_2, \cB_2)$. Let $\varphi_1 \colon \cB_1 \to \bR$ and $\varphi_2 \colon \cB_2 \to \bR$ be set functions. We say that a function $\varphi \colon \cB \to \bR$ is a {\it coupling} of $\varphi_1$ and $\varphi_2$ if $\varphi(X_1 \times J_2) = \varphi_1(X_1) \cdot \varphi_2(J_2)$ for every $X_1 \in \cB_1$ and $\varphi(J_1 \times X_2) = \varphi_1(J_1) \cdot \varphi_2(X_2)$ for every $X_2 \in \cB_2$. Similarly to the finite case, for $e_1\in J_1$, $e_2\in J_2$, and $Z\in \cB$, we use $Z_{e_1}$ and $Z^{e_2}$ to denote the $e_1$- and $e_2$-fibers of $Z$, and $\pi_1, \pi_2$ to denote the coordinate maps. By $\mathbbm{1}_Z$ we denote the indicator function of a set $Z$.

The main results of \cref{sec:submod} and \cref{sec:polymatroid} can be extended to pairs of bounded submodular set functions defined on infinite set algebras. Essentially the same ideas work as in the finite cases, i.e., in the proofs of \cref{thm:submod} and \cref{thm:poly}, but the reasoning is slightly more technical, and we need to integrate instead of taking sums. Moreover, we will need to integrate with respect to nonnegative charges. We follow the definitions of \cite{lovasz2023submodular}. Let $(J, \cB)$ be a set algebra, $\alpha \colon \cB \to \bR_{\ge 0}$ be a nonnegative charge, and $f \colon J \to \bR_{\ge 0}$ be a nonnegative bounded measurable function. As in~\cite{lovasz2023submodular}, we use the shorthand $\{f\ge t\}$ for the level set $\{x\in J\colon f(x)\ge t\}$. Define
\begin{equation*}
    \hat{\alpha}(f)\coloneqq\int_{0}^{\infty} \alpha(\{f\ge t\}) \, \diff t.
\end{equation*}
Note that $\hat\alpha$ is {\it linear} in the sense that $\hat{\alpha}(c_1f_1+c_2f_2)= c_1\hat\alpha(f_1)+c_2\hat\alpha(f_2)$ for all nonnegative bounded measurable functions $f_1, f_2$, and $c_1, c_2 \in \bR_{\ge 0}$. Furthermore, $\hat\alpha$ is {\it positive}, meaning that $\hat{\alpha}(g) \ge \hat{\alpha}(f)$ if $g(x) \ge f(x)$ for all $x \in J$.

\begin{rem}
    The functional $\hat\alpha$ is the Choquet extension of $\alpha$. It extends $\alpha$ from indicator functions to bounded measurable functions, since $\hat\alpha(\mathbf 1_X)=\alpha(X)$; in particular, if $\alpha$ is a probability measure, then $\hat\alpha(f)$ is the usual expectation $\mathbb E_\alpha[f]$. In \cite{lovasz2023submodular}, the integral is defined for any set function $\alpha$ with bounded variation, and any bounded measurable function $f$. 
\end{rem}

Let $(J_1,\cB_1)$ and $(J_2,\cB_2)$ be set algebras and denote their product space by $(J, \cB)$.
Let $\varphi_1 \colon \cB_1 \to \bR_{\ge 0}$, $\varphi_2 \colon \cB_2 \to \bR_{\ge 0}$ be bounded submodular functions satisfying $\varphi_1(\emptyset)=\varphi_2(\emptyset)=0$, and let $\alpha_1 \colon \cB_1 \to \bR_{\ge 0}$, $\alpha_2 \colon \cB_2 \to \bR_{\ge 0}$ be charges. To make it easier to draw parallels with the finite case, for any set $Z\in\cB$, by a slight abuse of notation, we denote by $\varphi_2(\pi_2(Z_{e_1}))$ the function $e_1 \mapsto\varphi_2(\pi_2(Z_{e_1}))$, by $\varphi_1(\pi_1(Z^{e_2}))$ the function $e_2 \mapsto\varphi_1(\pi_1(Z^{e_2}))$, and by $\alpha_2(\pi_2(Z_{e_1}))$ the function $e_1 \mapsto \alpha_2(\pi_2(Z_{e_1}))$. Any set $Z$ in the product algebra has the form $Z=\cup_{i=1}^n A_i\times B_i$, where $A_i\in \cB_1$ and $B_i\in\cB_2$ for $i\in[n]$. Consider the finite algebras $\cA_1$, $\cA_2$ generated by the sets $\{A_i\}_{i=1}^n$ and the sets $\{B_i\}_{i=1}^n$, respectively. It is not difficult to see that $\varphi_2(\pi_2(Z_{e_1}))$ is constant on the atoms of $\cA_1$ and $\varphi_1(\pi_1(Z^{e_2}))$ is constant on the atoms of $\cA_2$. Consequently, both functions take only finitely many values, and the preimages of these values belong to $\cA_1$ and $\cA_2$, respectively. As $\cA_i\subseteq \cB_i$ for $i=1,2$,  we conclude that both functions are measurable. Similarly to \eqref{eq:submod_coupling}, we can define
\begin{equation} \label{eq:infsubmod_coupling}
    b(Z) \coloneqq \hat{\alpha}_1 (\varphi_2(\pi_2(Z_{e_1})))+\hat{\alpha}_2(\varphi_1(\pi_1(Z^{e_2}))) -\hat{\alpha}_1 (\alpha_2(\pi_2(Z_{e_1})))
\end{equation}
for all $Z \in \cB$. While this definition does not look symmetric at first, note that $\hat{\alpha}_1 (\alpha_2(\pi_2(Z_{e_1})))$ and $\hat{\alpha}_2 (\alpha_1(\pi_1(Z^{e_2})))$ agree on the semi-algebra of product sets and thus agree on the generated algebra; see e.g.~\cite[Theorem 3.5.1 (ii)]{rao1983theory}. Observe that if $Z=Y_1 \times Y_2$ for some $Y_1 \in \cB_1$ and $Y_2 \in \cB_2$, then
\begin{align*}
    \hat{\alpha}_1 (\varphi_2(\pi_2(Z_{e_1})))&=\alpha_1(Y_1)\varphi_2(Y_2), \\
    \hat{\alpha}_2 (\varphi_1(\pi_1(Z^{e_2})))&=\varphi_1(Y_1)\alpha_2(Y_2), \\
    \hat{\alpha}_1 (\alpha_2(\pi_2(Z_{e_1})))&=\alpha_1(Y_1)\alpha_2(Y_2).
\end{align*}

First we extend the statement of \cref{thm:submod}. 

\begin{thm} \label{thm:inf1}  
Let $\varphi_1, \varphi_2$ be nonnegative submodular functions over the set algebras $(J_1, \cB_1)$ and $(J_2, \cB_2)$, respectively, satisfying $\varphi_1(\emptyset)=0$ and $\varphi_2(\emptyset)=0$. Furthermore, let $\alpha_1\colon \cB_1 \to \bR_{\ge 0}$ and $\alpha_2\colon \cB_2 \to \bR_{\ge 0}$ be charges such that $\alpha_1(J_1)=\varphi_1(J_1)$ and $\alpha_2(J_2)=\varphi_2(J_2)$. Then, the function $b$ defined in~\eqref{eq:infsubmod_coupling} is a submodular coupling of $\varphi_1$ and $\varphi_2$.
\end{thm}
\begin{proof} By the definition of $b$, we have
\begin{align*} 
    b(Y_1\times Y_2)
    &=\alpha_1(Y_1)\cdot\varphi_2(Y_2)+\alpha_2(Y_2)\cdot\varphi_1(Y_1)-\alpha_1(Y_1)\cdot\alpha_2(Y_2) \\
    &=\varphi_1(Y_1)\cdot\varphi_2(Y_2)- (\varphi_1(Y_1)-\alpha_1(Y_1))\cdot (\varphi_2(Y_2)-\alpha_2(Y_2)) \\
    &=\varphi_1(Y_1)\cdot\varphi_2(Y_2),
\end{align*}
if $Y_1=J_1$ or $Y_2=J_2$; hence $b$ is a coupling.
To prove submodularity, note that for each $e_1 \in J_1$ and $e_2 \in J_2$, the set functions $\cB\to \bR_{\ge 0}$ defined by $Z\mapsto \varphi_2(\pi_2(Z_{e_1}))$ and $Z\mapsto \varphi_1(\pi_1(Z^{e_2}))$ are submodular, while $Z\mapsto \alpha_2(\pi_2(Z_{e_1}))$ is modular; hence $b$ is submodular by the linearity and positivity of $\hat{\alpha}_1$ and $\hat{\alpha}_2$.
\end{proof}

To extend the result of \cref{thm:poly}, we first need to define the measurable analogue of a base polytope. Let $(J, \cB)$ be a set algebra and $\varphi \colon \cB \to \bR_{\ge 0}$ be a polymatroid function. We say that a nonnegative charge $\alpha \colon \cB \to \bR_{\ge 0}$ is a {\it minorizing charge} if $\alpha(A)\le\varphi(A)$ for all $A\in\cB$. A minorizing charge is called {\it basic} if $\alpha(J)=\varphi(J)$. The set of all basic minorizing charges is denoted by $\bmm(\varphi)$. It was proved in~\cite[Corollary 5.3]{lovasz2023submodular} that $\bmm(\varphi)$ is nonempty. 

\begin{thm} \label{thm:inf2}  
Let $\varphi_1$, $\varphi_2$ be polymatroid functions on measurable spaces $(J_1,\cB_1)$ and $(J_2,\cB_2)$, respectively. Then, there exists a polymatroid coupling $\varphi$ of $\varphi_1$ and $\varphi_2$ on the product set algebra $(J, \cB)$.
\end{thm}
\begin{proof}
    Choose $\alpha_1\in \bmm(\varphi_1)$ and $\alpha_2\in \bmm(\varphi_2)$, and define $b$ as in \eqref{eq:infsubmod_coupling}. 

\begin{cl} \label{cl:infb} 
We have $b(\emptyset) = 0$ and $b(Y_1 \times Y_2) \le \varphi_1(Y_1)\cdot  \varphi_2(Y_2)$  for each $Y_1 \in \cB_1$ and $Y_2 \in \cB_2$, where equality holds if $\varphi_1(Y_1) =\alpha_1(Y_1)$ or $\varphi_2(Y_2)=\alpha_2(Y_2)$. 
\end{cl}

\begin{proof}
The equality $b(\emptyset) = 0$ follows from $\varphi_1(\emptyset)=\varphi_2(\emptyset) = 0$.
For each $Y_1 \in \cB_1$ and $Y_2 \in \cB_2$, using $\alpha_1(Y_1) \le \varphi_1(Y_1)$ and $\alpha_2(Y_2) \le \varphi_2(Y_2)$, we get
\[
b(Y_1 \times Y_2) =
\varphi_1(Y_1)\cdot\varphi_2(Y_2)- (\varphi_1(Y_1)-\alpha_1(Y_1))\cdot (\varphi_2(Y_2)-\alpha_2(Y_2)) \le 
\varphi_1(Y_1)\cdot  \varphi_2(Y_2). \]
Here equality holds if $\varphi_1(Y_1) =\alpha_1(Y_1)$ or $\varphi_2(Y_2)=\alpha_2(Y_2)$.
\end{proof}

The function $b$ is not necessarily monotone. Let $\varphi(Z)\coloneqq \inf_{Z'\supseteq Z} b(Z')$ for all $Z \in \cB$. 

\begin{cl}\label{cl:infinc}
$\varphi$ is increasing and submodular.
\end{cl}
\begin{proof}
Clearly, $\varphi$ is increasing by definition. To prove submodularity, let $U,V\in \cB$. Then, for any $\varepsilon>0$, there exist $U'\supseteq U$ and $V'\supseteq V$ such that $\varphi(U)+\varepsilon \ge b(U')$ and $\varphi(V)+\varepsilon \ge b(V')$. We get 
\begin{align*}
\varphi(U)+\varphi(V)
&\ge
b(U')+b(V')-2\varepsilon\\
&\ge 
b(U'\cap V')+b(U'\cup V')-2\varepsilon\\
&\ge
\varphi(U\cap V)+\varphi(U\cup V)-2\varepsilon.
\end{align*}
The claim follows by taking $\varepsilon \to 0$.
\end{proof}

\begin{cl} \label{cl:infvarphi_on_product}
$\varphi(Y_1\times Y_2) = b(Y_1\times Y_2)$ holds for each $Y_1\in\cB_1$ and $Y_2\in\cB_2$.
\end{cl}

\begin{proof}
    We need to show that $b(Z) \ge b(Y_1\times Y_2)$ holds if $Y_1\times Y_2 \subseteq Z \in \cB$. Using that $\varphi_2$ is increasing and $\alpha_2(X) \le \varphi_2(X)$ holds for any $X\in \cB_2$, we get
    \begin{align*}
        \hat{\alpha}_1 (\varphi_2(\pi_2(Z_{e_1})))&=\hat{\alpha}_1 (\mathbbm{1}_{Y_1} \cdot \varphi_2(\pi_2(Z_{e_1})))+\hat{\alpha}_1 (\mathbbm{1}_{J_1 \setminus Y_1} \cdot \varphi_2(\pi_2(Z_{e_1}))) \\
        &\ge \alpha_1(Y_1)\cdot \varphi_2(Y_2)+\hat{\alpha}_1 (\alpha_2(\pi_2((Z\setminus (Y_1 \times J_2))_{e_1}))).
    \end{align*}
    Similarly, 
    \begin{align*}
        \hat{\alpha}_2 (\varphi_1(\pi_1(Z^{e_2})))&=\hat{\alpha}_2 (\mathbbm{1}_{Y_2} \cdot \varphi_1(\pi_1(Z^{e_2})))+\hat{\alpha}_2 (\mathbbm{1}_{J_2 \setminus Y_2} \cdot \varphi_1(\pi_1(Z^{e_2}))) \\
        &\ge \varphi_1(Y_1)\cdot \alpha_2(Y_2)+\hat{\alpha}_2 (\mathbbm{1}_{J_2 \setminus Y_2} \cdot \alpha_1(\pi_1(Z^{e_2}))) \\
        &=\varphi_1(Y_1)\cdot \alpha_2(Y_2)+\hat{\alpha}_1 (\alpha_2(\pi_2((Z\setminus (J_1 \times Y_2))_{e_1}))).
    \end{align*}
    Finally, 
    \begin{align*}
        \hat{\alpha}_1 (\alpha_2(\pi_2(Z_{e_1})))
        &=\alpha_1(Y_1)\cdot \alpha_2(Y_2)+\hat{\alpha}_1 (\alpha_2(\pi_2((Z\setminus (Y_1 \times Y_2))_{e_1}))) \\
        &\le \alpha_1(Y_1)\cdot \alpha_2(Y_2)+\hat{\alpha}_1 (\alpha_2(\pi_2((Z\setminus (Y_1 \times J_2))_{e_1})))+\hat{\alpha}_1 (\alpha_2(\pi_2((Z\setminus (J_1 \times Y_2))_{e_1}))).
    \end{align*}
    The claim follows by combining these three inequalities.
\end{proof}

Combining Claims~\ref{cl:infb},~\ref{cl:infinc}, and~\ref{cl:infvarphi_on_product}, we deduce that $\varphi$ is a polymatroid coupling of $\varphi_1$ and $\varphi_2$. This concludes the proof of the theorem.
\end{proof}

\section{Universal functions}
\label{sec:universal}

In this section we show that, under certain conditions, if there exists a coupling of any two functions with some property $P$, then there exists a universal function with property $P$, i.e., a function that contains any finite function with property $P$ as a quotient. 

\subsection{Preparations}
\label{sec:prep}

If $\cQ = (Q_1, \dots, Q_q)$ is a finite measurable partition of $J$ into $q$ possibly empty parts, i.e.,  $(Q_1, \dots, Q_q)$ is a partition and $Q_i \in \cB$ for $i \in [q]$, then the {\it quotient} $\varphi/\cQ\colon 2^{[q]} \to \bR$ denotes the function defined by $(\varphi/\cQ)(X) = \varphi(\bigcup_{i \in X} Q_i)$ for $X \subseteq [q]$. The notion of coupling can be naturally extended to more than two functions. Let $(J_1,\cB_1),\dots,(J_n,\cB_n)$ be set algebras and denote their product set algebra by $(J,\cB)$. Furthermore, let $\varphi_i\colon\cB_i\to\bR$ be set functions for $i\in[n]$. We say that a function $\varphi\colon\cB\to\bR$ is a {\it coupling} of $\varphi_1,\dots,\varphi_n$ if $\varphi(J_1\times\dots\times J_{i-1}\times X_i\times J_{i+1}\times \dots\times J_n)=\varphi_i(X_i)\cdot\prod_{\ell\in [n]\setminus\{i\}}\varphi_\ell(J_\ell)$ for every $X_i\in\cB_i$, $i\in [n]$.

Let $(J_1, \cB_1)$ and $(J_2, \cB_2)$ be two set algebras. The functions $\varphi_1\colon \cB_1 \to \bR$ and $\varphi_2\colon \cB_2 \to \bR$ are {\it isomorphic} if there exists a bijection $f\colon J_1 \to J_2$ such that $f$ and $f^{-1}$ are measurable and $\varphi_2(f(X))=\varphi_1(X)$ for all $X\in \cB_1$.
By a {\it property} of set functions, we mean an isomorphism-invariant class of set functions.
We list five conditions on a property $P$ of set functions that we will assume later. 
\begin{enumerate}[label=(\roman*)] \itemsep0em
\item  $P$ is {\it finitary} if a function $\varphi\colon \cB \to \bR$ is in $P$ if and only if $\varphi/\cQ$  is in $P$ for every finite measurable partition $\cQ$ of $J$.
\item A finitary property $P$ is {\it closed} if for every measurable space $(J, \cB)$ and every finite measurable partition $\cQ$ of $J$, the set of functions $\{f\colon \cB \to \bR \mid f/\cQ \in P\}$ is closed. 
\item A finitary property $P$ is {\it extendable} if for every $\varphi\colon2^{[q]} \to \bR$ with property $P$ and any finite, pairwise disjoint nonempty sets $A_1, \ldots, A_q$ with union $A$, there exists a function $\psi \colon 2^A \to \bR$ with property $P$ that extends $\varphi$ in the sense that $\varphi(X)=\psi(\bigcup_{i \in X} A_i)$ for all $X \subseteq [q]$.
\item A finitary property $P$ is {\it liftable} if any two finite functions with property $P$ admit a coupling with property $P$ on the product of their ground sets. Note that this implies that any finite number of finite functions with property $P$ admit a coupling with property $P$ on the product of their ground sets.
\item For $b \in \bR_{\ge 0}$, property $P$ is {\it $b$-bounded} if the range of each function with property $P$ is contained in $[-b,b]$.
\end{enumerate}
Finally, a normalized function with property $P$ is called {\it universal for property $P$} if it contains every finite, normalized function with property $P$ as a quotient.

We need the following well-known lemma (which follows from e.g.\ \cite[Problem~4.43]{komjath2006problems}) and include a proof for completeness.

\begin{lem} \label{lem:independent}
    There exists a family of subsets $\cH \subseteq 2^\bZ$ with $|\cH|=\fc$ such that for any disjoint finite subsets $\cH_1, \cH_2 \subseteq \cH$, we have
    \begin{equation*}
        \bigcap_{H \in \cH_1} H \cap \bigcap_{H \in \cH_2} H^c \neq \emptyset.
    \end{equation*}
\end{lem}
\begin{proof}
    Let 
    \begin{equation*}
        \cR\coloneqq \bigl\{ A \subseteq \bR \mid A=\bigcup_{i=1}^n [p_i, q_i] \text{ for some $n \in \bZ_{\ge 0}$ and rational numbers $p_i, q_i$ for $i \in [n]$} \bigr\}.
    \end{equation*} 
    For any $x \in \bR$, let $\cH_x\coloneqq \{A \in \cR \mid x \in A \}$. For any disjoint finite subsets $X_1 \subseteq \bR$ and $X_2 \subseteq \bR$ we can clearly find $A \in \cR$ with $X_1 \subseteq A$ and $X_2 \cap A=\emptyset$. Finally, since $\mathcal R$ is countable, choose a bijection $\eta\colon \mathbb Z\to\mathcal R$ and replace each $H_x$ by $\eta^{-1}(H_x)\subseteq\mathbb Z$. This gives the required family in $2^{\mathbb Z}$.
\end{proof}

The proof of \cref{thm:countable} relies on a compactness argument. For a set algebra $(J, \cB)$, the space of set functions $\{f\colon \cB \to \bR\}$ can be identified with the product topological space $\prod_{A \in \cB} \bR$. For any $r\colon \cB \to \bR_{\ge 0}$, the space of functions 
\begin{equation*}
   \bigl\{f\colon \cB \to \bR \ \big| \  |f(A)| \le r(A) \text{ for all } A \in \cB \bigr\}=\prod_{A \in \cB} [-r(A),r(A)] 
\end{equation*} 
is compact, as the product of compact sets is compact by Tychonoff's theorem \cite{Tychonoff1935}. We will use that a topological space $X$ is compact if and only if $\bigcap_{C \in \cC} C \neq \emptyset$ for any collection $\cC$ of closed subsets of $X$ such that $\bigcap_{C \in \cC'} C \neq \emptyset$ for all finite subcollections $\cC'\subseteq \cC$.

\subsection{Proof of existence}
\label{sec:existence}

Now we are ready to state the main result of this section.

\begin{thm}\label{thm:countable}
    For any finitary, closed, extendable, liftable, and $b$-bounded property $P$ for some $b \ge 1$, there exists a universal function on the set algebra $(\bZ, 2^\bZ)$.
\end{thm}
\begin{proof}
    Let $\cH \subseteq 2^\bZ$ be a set given by \cref{lem:independent}. Let $\Phi$ be a set of representatives of normalized functions with property $P$ on finite ground sets whose sizes are powers of two, that is, for each such isomorphism class $\Phi$ contains exactly one function. For a function $\varphi\in\Phi$, denote by $S_\varphi$ the ground set of $\varphi$. We may assume that for any two distinct $\varphi,\psi\in\Phi$, their ground sets are disjoint, that is, $S_\varphi\cap S_\psi=\emptyset$. Since $P$ is extendable, it is enough to find a function $F \colon 2^\bZ \to \bR$ with property $P$ that contains every member of $\Phi$ as a quotient.

    Up to isomorphism, there are continuum many real-valued functions on finite ground sets, hence $|\Phi| \le \fc$. Consequently, we can choose $H_1^\varphi, \dots , H_n^\varphi \in \cH$ for all $\varphi \in \Phi$, where $n=\log_2(|S_\varphi|)$, such that $H_i^{\varphi_1}=H_j^{\varphi_2}$ only if $\varphi_1=\varphi_2$ and $i=j$. Furthermore, for all $\varphi \in \Phi$ and all $a \in S_\varphi$, define $I_{a} \subseteq [n]$ and 
    \begin{equation*}
    T_{a}\coloneqq \bigcap_{i \in I_{a}} H_i^\varphi \cap \bigcap_{i \in [n] \setminus I_{a}} \left(H_i^\varphi\right)^c,
    \end{equation*}
    where $n=\log_2(|S_\varphi|)$, such that $I_{a} \neq I_{b}$ for distinct $a,b \in S_\varphi$. It is enough to prove the existence of a function $F \colon 2^\bZ \to \bR$ such that $F$ has property $P$ (which is equivalent to every finite quotient having property $P$) and
     \begin{equation*}
        F\bigl(\bigcup_{a \in A} T_{a}\bigr)= \varphi(A)
    \end{equation*}
    for all $\varphi \in \Phi$ and $A \subseteq S_\varphi$, as the latter condition guarantees that it contains every finite function with property $P$ as a quotient. Note that, by the property of the family $\cH$ given by \cref{lem:independent}, $\bigcup_{a \in A} T_{a} \neq \bigcup_{b \in B} T_{b}$ for any $\varphi, \psi \in \Phi$ and nonempty subsets $A \subsetneq S_\varphi$, $B \subsetneq S_\psi$. 

    The function space $\{f\colon 2^\bZ \to \bR \mid |f(A)| \le b \text{ for all $A \in 2^\bZ$}
    \}$ is compact. The condition that $F/\cQ$ has property $P$ is closed for every finite measurable partition $\cQ$ of $\bZ$. Furthermore, conditions of the form $F(\bigcup_{a \in A} T_{a})= \varphi(A)$ are also closed. Hence, it is enough to prove that for any finite number of these conditions, we can find a function satisfying all of them. Let $\cQ_1, \ldots , \cQ_n$ be finite partitions of $\bZ$, and let $\bigcup_{a \in A_i} T_a$ be sets for some $A_1, \ldots , A_m$, where there exist distinct $\varphi_1, \ldots , \varphi_k \in \Phi$ such that $A_i \subseteq S_{\varphi_{f(i)}}$ for all $1 \le i \le m$ and some function $f \colon [m] \to [k]$. We need to show that there exists a function $F' \colon 2^\bZ \to \bR$ such that $F'(\bigcup_{a \in A_i} T_a)=\varphi_{f(i)}(A_i)$ for $1 \le i \le m$, and $F'/\cQ_i$ has property $P$ for $1 \le i \le n$.  
    
    Let $\varphi$ be a coupling of $\varphi_1,\ldots, \varphi_k$ on ground set $S \coloneqq S_{\varphi_1} \times  \ldots \times S_{\varphi_k}$. Denote by $\cA_1$ the algebra generated by the sets $\{H_j^{\varphi_i} \mid 1 \le i \le k, \ 1 \le j \le \log_2(|S_{\varphi_i}|)\}$. The atoms of $\cA_1$ are of the form 
    $\bigcap_{i=1}^k T_{a_i}$ for some $(a_1, \ldots , a_k) \in S$. Note that none of these atoms are empty by \cref{lem:independent}. The set algebras $2^S$ and $\cA_1$ are isomorphic, as to an atom $(a_1, \ldots , a_k) \in S$ we can assign $\bigcap_{i=1}^k T_{a_i}$ of $\cA_1$, which induces the isomorphism $\eta \colon 2^S \to \cA_1$,
    \begin{equation*}
        \eta(A)=\bigcup_{(a_i)_{i=1}^k \in A} \bigl(\bigcap_{i=1}^k T_{a_i}\bigr)
    \end{equation*}
    for all $A \in 2^S$. Now we can define the function $F_1 \colon \cA_1 \to \bR$ by $F_1(A)=\varphi(\eta^{-1}(A))$ for all $A \in \cA_1$. By definition, this satisfies all equalities of the form $F_1(\bigcup_{a \in A_i} T_a)=\varphi_{f(i)}(A_i)$, and $F_1$ has property $P$ as $\varphi$ has property $P$. Let $\cA_2$ be the algebra generated by the sets of $\cA_1$ and the sets appearing in one of the partitions $\cQ_i$ for $1 \le i \le n$. This is a refinement of $\cA_1$ still having a finite number of atoms. Hence, by extendability, there exists a function $F_2 \colon \cA_2 \to \bR$ with property $P$ that extends $F_1$, i.e., $F_2|_{\cA_1}=F_1$. It follows that $F_2/\cQ_i$ has property $P$ for $1 \le i \le n$. Finally, an arbitrary extension $F'$ of $F_2$ to the domain $2^\bZ$ concludes the proof of the theorem.
\end{proof}

\subsection{Applications}
\label{sec:applications}

We show two applications of \cref{thm:countable}. First, we show that there exists a universal increasing submodular set function, using \cref{thm:poly}. Any increasing submodular set function with value $0$ on the empty set, other than the constant zero function, can be normalized. Hence, such a universal function contains every increasing submodular set function with value $0$ on the empty set as a quotient, up to a constant multiplier.

\begin{cor}\label{cor:univ_increasing_submod}
    There exists a universal increasing submodular set function. 
\end{cor}
\begin{proof}
    Let $P$ be the property of being normalized, increasing, and submodular. To apply \cref{thm:countable}, we need to show that $P$ is a finitary, closed, extendable, liftable, and 1-bounded property.
    For a set function $\varphi\colon\cB \to \bR$ on $(J, \cB)$, having property $P$ means being normalized and satisfying
    \begin{equation}\label{eq:monsubmod}
        \varphi(A) \le \varphi(B) \ \ \text{ and } \ \ \varphi(C \cap D)+\varphi(C \cup D) \le \varphi(C)+\varphi(D)
    \end{equation}
    for all $A, B, C, D \in \cB$ with $A \subseteq B$. Hence, it is clear that $P$ is finitary, and 1-boundedness also holds by monotonicity and $\varphi(J)=1$. Liftability follows from \cref{thm:poly}. Note that for each inequality in \eqref{eq:monsubmod}, the set of functions $f\colon \cB\to \bR$ satisfying it is closed, and so is the set of normalized functions. The set of functions with property $P$ is the intersection of these sets and is therefore closed.
    Consequently,
    $\{f\colon \cB \to \bR \mid f/\cQ \in P\}$ is closed for all finite measurable partitions, as it is a projection of $\{f\colon \cB \to \bR\}$ to the algebra generated by $\cQ$, proving that $P$ is closed. 
    Finally, if $\varphi\colon2^{[n]} \to \bR$ has property $P$, and $A_1, A_2, \ldots , A_n$ are finite, pairwise disjoint sets with union $A$, then the set function $\psi\colon 2^A \to \bR$ defined by $\psi(X) \coloneqq\varphi(\nu(X))$ for all $X \subseteq A$, where $\nu(X) \coloneqq \{ i \in [n] \mid X \cap A_i \neq \emptyset \}$ for all $X \subseteq A$, is an extension with property $P$. Indeed, $\psi$ is clearly an increasing extension, and for any $X, Y \subseteq A$, 
    \begin{align*}
        \psi(X)+\psi(Y)&=\varphi(\nu(X))+\varphi(\nu(Y)) \\ 
        &\ge \varphi(\nu(X) \cap \nu(Y))+\varphi(\nu(X) \cup \nu(Y)) \\ 
        &\ge \varphi(\nu(X \cap Y))+\varphi(\nu(X\cup Y)) \\
        &=\psi(X \cap Y)+\psi(X \cup Y).
    \end{align*}
     Hence property $P$ is extendable as well, so we can apply \cref{thm:countable} to finish the proof of the corollary.
\end{proof}

In particular, \cref{cor:univ_increasing_submod} implies that there exists a function $F\colon 2^\bZ \to \bR$ that contains every normalized matroid rank function as a quotient.

As a second application, we use \cref{cor:inf} to show that there exists a universal function for coverage functions. The proof relying on \cref{thm:countable} is only existential; later, in \cref{sec:infalt}, we describe an explicit construction.

\begin{cor}\label{cor:infalt}
    There exists a universal coverage function. 
\end{cor}
\begin{proof}
    The proof is similar to that of \cref{cor:univ_increasing_submod}. The coverage property is given by the inequalities \ref{eq:kalt}; hence, this property is finitary and closed. By monotonicity, it is also 1-bounded. \cref{cor:inf} gives liftability. It remains to prove that if $\varphi\colon2^{[n]} \to \bR$ is a coverage function and $A_1, A_2, \ldots , A_n$ are finite, pairwise disjoint sets with union $A$, then we can extend $\varphi$ to a set function on $A$. By \cref{prop:extremal}, we can write 
    \begin{equation*}
        \varphi=\sum_{\substack{\emptyset\neq X \subseteq [n]}} c_X\varphi_X,
    \end{equation*} 
    where $\varphi_X$ is defined as in \eqref{eq:extremal_def} and $c_X \ge 0$ for all $\emptyset \neq X \subseteq [n]$. Let $u(X)\coloneqq\bigcup_{i \in X} A_i$ for all $X \subseteq [n]$ and define 
    \begin{equation*}
        \psi \coloneqq \sum_{\substack{\emptyset\neq X \subseteq [n]}} c_X\varphi_{u(X)}.
    \end{equation*}
    Then $\psi$ is a coverage function by \cref{prop:extremal}, and it clearly extends $\varphi$, proving extendability. Hence the corollary follows by \cref{thm:countable}.
\end{proof}

\begin{rem}
Note that the existence of a universal function for $b$-bounded submodular functions does not follow from our results, as \cref{thm:submod} does not ensure that the coupling of $b$-bounded submodular functions remains $b$-bounded. Furthermore, even achieving extendability is nontrivial, as it is unclear how to extend a $b$-bounded submodular function while preserving the bound.
\end{rem}

\subsection{Construction for coverage functions}
\label{sec:infalt}

By Corollary~\ref{cor:infalt}, finite normalized coverage functions admit a universal function. However, the proof of Theorem~\ref{thm:countable} offers little insight into the form of such a function. In this section, we construct an explicit coverage function so that any finite normalized coverage function can be obtained as its quotient.

\begin{thm}\label{thm:construction}
Let $\cB$ be the Borel $\sigma$-algebra on $\bR$ and $\lambda$ be the Lebesgue measure. For any $A\in\cB$, define $\Phi(A)=\lambda(\{\fr(a)\mid a\in A\})$. Then, $\Phi$ is a universal coverage function.
\end{thm}
\begin{proof}
    Note that for any $A\in\cB$, we have $\{\fr(a)\mid a\in A\}=\bigcup_{k\in\bZ}\{a-k\mid a\in A\cap[k,k+1)\}$, which is a measurable set; hence, $\Phi(A)$ is well-defined.
    First we show that $\Phi$ is indeed a coverage function. To see this, let $A_0,A_1,\dots,A_k\in\cB$ be arbitrary. Define $\hat{A_i}\coloneqq \{\fr(a)\mid a\in A_i\}$ for $i=0,1,\dots,k$. Then, \ref{eq:kalt} for $\Phi$ and sets $A_0,A_1,\dots,A_k$ is equivalent to \ref{eq:kalt} for $\lambda$ and sets $\hat{A_0},\hat{A_1},\dots,\hat{A_k}$. Since any measure is a coverage function, the same holds for $\Phi$. 
    
    Now let $q$ be a positive integer and $\psi\colon2^{[q]}\to \bR$ be a normalized coverage function, that is, $\psi(\emptyset)=0$, $\psi([q])=1$, and $\psi$ satisfies \ref{eq:kalt} for all positive integers $k$ and choices of $A_0,A_1,\dots,A_k\in2^{[q]}$. Since the function $\psi$ is increasing and normalized, we have $0\le \psi\le 1$. For ease of discussion, we use the notation $f(A)\coloneqq \{\fr(a)\mid a\in A\}$.  Split the range $[0,1)$ into $q$ intervals $L_1,\dots,L_q$ where $L_i = [\psi([i-1]), \psi([i]))$ for $i\in[q]$. Then, for each $i\in[q]$, the set $\fr^{-1}(L_i)$ consists of infinitely many intervals $\{J_i^m\}_{m=1}^{\infty}$ where, for ease of discussion, we define $J_i^m=\emptyset$ for $m\in\bZ_{>0}$ if $\psi([i-1])=\psi([i])$.
    
    We construct a measurable partition $\cQ=(Q_1,\dots,Q_q)$ of $\bR$ for which $\psi=\Phi/\cQ$ in $q$ successive steps. In Step $t$, we construct a function $\phi_t$ over $2^{[q]}$ satisfying the following:
    \begin{enumerate}[label=(\Alph*)]\itemsep0em
        \item $\phi_t(I)=\psi(I)$ if $I\subseteq [t]$, $t\in I$, \label{it:a}
        \item $\phi_{t-1}(I)=\phi_t(I)$ if $I\subseteq [t-1]$. \label{it:b}
    \end{enumerate} 
    Note that \ref{it:a} and \ref{it:b} together imply $\phi_q(I)=\psi(I)$ whenever $I\subseteq[q]$ is nonempty, while $\phi_q(\emptyset)=0$ will follow easily from the construction. Indeed, let $I\subseteq [q]$ be an arbitrary nonempty set. If $\max I=q$, then $\phi_q(I)=\psi(I)$ holds by \ref{it:a}. Otherwise, we have $\phi_q(I)=\phi_{\max I}(I)=\psi(I)$, where the first equality follows by \ref{it:b} while the second follows by \ref{it:a}. 
    
    In Step 1, we set $Q_i$ to be $J_i^1$ for each $i\in[q]$ and define $\phi_1(I)\coloneqq \Phi(\bigcup_{i\in I}Q_i)$ for $I\subseteq [q]$. The function $\phi_1$ thus obtained is modular and satisfies $\phi_1\le \psi$; in particular, if $\psi$ is modular then we have $\phi_1=\psi$. By construction, \ref{it:a} is satisfied as $\phi_1(\{1\})=\Phi(Q_1)=\Phi(J_1^1)=\psi(\{1\})$. In fact, $\phi_1([t])=\psi([t])$ holds for every $t\in[q]$.

    For $t\ge 2$, in Step $t$ we increase $Q_t$ by adding certain subsets $B_t^I\subseteq J^t_{\min I}$ for $\emptyset\neq I\subseteq[t-1]$, and then we define $\phi_t(I)\coloneqq \Phi(\bigcup_{i\in I}Q_i)$ for $I\subseteq [q]$. This guarantees that $\phi_{t-1}(I)=\phi_t(I)$ if $I\subseteq [t-1]$, that is, \ref{it:b} is satisfied. However, for \ref{it:a} to be satisfied as well, the sets $B_t^I$ must be chosen carefully.

    For $\emptyset\neq I\subseteq [t-1]$, let us define
    \begin{equation*}
    b_t^I = \sum_{J\subseteq I\cup \{t\}} (-1)^{|J|+1}\psi\big(([t-1]\setminus I) \cup J\big).
    \end{equation*}
    The role of these quantities is that eventually we will set $B_t^I$ such that $b_t^I=\Phi(B_t^I)$, but first, we prove two technical claims providing lower and upper bounds on $b_t^I$.

    \begin{cl}\label{cl:1}
        $b_t^I\ge 0$.   
    \end{cl}
    \begin{proof}
        The statement follows by \ref{eq:kalt} for $\psi$ and $k\coloneqq |I|+1$ with the choice $A_0 = [t-1]\setminus I$ and $A_1,\dots,A_k$ being the singletons of $I\cup \{t\}$.
    \end{proof}

    \begin{cl}\label{cl:2}
        $b_t^I\le b_{\max{I}}^{I\setminus \{\max I\}}-(b_{\max{I}+1}^I + ... + b_{t-1}^I)$.
    \end{cl}
    \begin{proof}
    By slight rearrangement and the definition of $b_t^I$, we need to show that
    \begin{equation*}
      \sum_{J\subseteq I} (-1)^{|J|+1}\psi\big(([\max I]\setminus I) \cup J\big)\ge \sum_{i=\max I +1}^t\sum_{J\subseteq I\cup \{i\}}(-1)^{|J|+1}\psi\big(([i-1]\setminus I)\cup J\big).  
    \end{equation*}
    This can be vastly simplified, as the right-hand side is a telescoping sum. Indeed, if for a given choice of $i=i_0$ we have a term $\pm \psi(([i_0-1]\setminus I)\cup J)$ with $i_0\in J$, then for $i=i_0+1$ we encounter the same term if we drop $i_0$ from $J$, with the opposite sign. This shows that all terms cancel except for those with $i=\max I +1$ and $i\notin J$, and those with $i=t$ and $i\in J$. However, the terms of the first kind correspond precisely to the left-hand side; hence that cancels as well. Thus the inequality simplifies to
    \begin{equation*}
        0\ge \sum_{J\subseteq I} (-1)^{|J|}\psi\big(([t]\setminus I)\cup J\big).
    \end{equation*}
    The latter inequality follows from \ref{eq:kalt} for $\psi$ and $k\coloneqq |I|$ with the choice $A_0 = [t]\setminus I$ and $A_1,\dots,A_k$ being the singletons of $I$. 
    \end{proof}

    With the help of Claims~\ref{cl:1} and~\ref{cl:2}, we show that there exist subsets $B_t^I\subseteq J_{\min I}^t$ for $\emptyset\neq I\subseteq [t-1]$ satisfying $b_t^I=\Phi(B_t^I)$ and being located so that the growth $\phi_t(I)-\phi_{t-1}(I)$ for $I\subseteq [t]$, $t \in I$ is precisely what is needed for~\ref{it:a}, as will be shown in Claim~\ref{cl:3}. For a set $X\subseteq [0,1)$, we denote by $X^c$ the complement of $X$, that is, $X^c=[0,1)\setminus X$.
    
    \begin{cl}\label{cl:4}
        For $\emptyset\neq I\subseteq [t-1]$, there exists $B_t^I\subseteq \fr^{-1}((\bigcap_{i\in I}f(Q_i))\cap (\bigcap_{i\in[t-1]\setminus I}f(Q_i)^c))\cap J^t_{\min I}$ such that $\Phi(B_t^I)=b_t^I$.
    \end{cl}
    \begin{proof}
    First observe that since $\Phi(B_t^I)=\lambda(f(B_t^I))\ge 0$, we need $b_t^I\ge 0$, which holds by Claim~\ref{cl:1}. On the other hand, we need to show that there is enough unallocated space in $J_{\min I}^t$ for the set $B_t^I$. That is, since for the desired $B_t^I$ we have $f(B_t^I) \subseteq (\bigcap_{i\in I}f(Q_i))\cap( \bigcap_{i\in[t-1]\setminus I}f(Q_i)^c)$, we need
    \begin{equation*}
        b_t^I \le \lambda\Big(\big(\bigcap_{i\in I}f(Q_i)\big)\cap \big(\bigcap_{i\in[t-1]\setminus I}f(Q_i)^c\big)\Big).    
    \end{equation*}    
    By definition, we have 
    \begin{equation*}
      b_{\max{I}}^{I\setminus \{\max I\}} = \lambda\Big(\big(\bigcap_{i\in I}f(Q_i)\big)\cap \big(\bigcap_{i\in[\max I]\setminus I}f(Q_i)^c\big)\Big),  
    \end{equation*}
    and the sets $B_{j}^I$ for $j=\max I +1, \dots, t-1$ have pairwise disjoint images in $f(B_{\max{I}}^{I\setminus \{\max I\}})$ of measure $b_j^I$. Thus, the necessary and sufficient upper bound on $b_t^I$ can be rewritten as 
    \begin{equation*}
    b_t^I\le b_{\max{I}}^{I\setminus \{\max I\}}-(b_{\max{I}+1}^I + ... + b_{t-1}^I),
    \end{equation*}
    which follows by Claim~\ref{cl:2}.    
    \end{proof}
    
    Let $B_t^I$ be sets provided by Claim~\ref{cl:4}.  We define $B_t\coloneqq\bigcup_{\emptyset\neq I\subseteq[t-1]} B_t^I$ and update $Q_t$ by adding $B_t$ to it, that is, $Q_t = J_t^1\cup B_t$ is the updated set.
    
    \begin{cl}\label{cl:3}
        For $I\subseteq[t-1]$, we have $\phi_t([t]\setminus I)=\psi([t]\setminus I)$.
    \end{cl}
    \begin{proof}
        We can express
        \begin{equation*}
        \Phi(\bigcup_{i\in [t]\setminus I}Q_i) = \Phi(\bigcup_{i\in [t-1]\setminus I} Q_i) + \Phi(J_t^1) + \sum_{\emptyset\neq J\subseteq I}\Phi(B_t^J). 
        \end{equation*}
        As $\Phi(J_t^1)=\psi([t])-\psi([t-1])$ by definition, $\phi_{t-1}([t-1]\setminus I)=\psi([t-1]\setminus I)$ by~\ref{it:a} and~\ref{it:b}, and $\Phi(B_t^J)=b_t^J$, we get that $\phi_t([t]\setminus I)=\psi([t]\setminus I)$ holds if and only if
        \begin{equation*}
            \psi([t]\setminus I)=\psi([t-1]\setminus I)+\psi([t])-\psi([t-1]) + \sum_{\emptyset\neq J\subseteq I}b_t^J.    
        \end{equation*}
        For $I=\emptyset$, the equation is satisfied due to Step 1. Otherwise, using the definition of $b_t^I$, one can count the occurrences of $\psi([t-1]\setminus I')$ and $\psi([t]\setminus I')$ for each $\emptyset\subseteq I' \subseteq I$ on both sides. As coefficients, we get sums of the form $\sum_{i=0}^{|I\setminus I'|}(-1)^i\binom{|I\setminus I'|}{i}$ which are $0$, except for the cases when $I'=\emptyset$ or $I'=I$. These exceptions yield the remainder terms $\psi([t]\setminus I)$, $\psi([t-1]\setminus I)$, $\psi([t])$, $\psi([t-1])$ with appropriate signs.
    \end{proof}
    
    By Claim~\ref{cl:3}, property \ref{it:a} is also satisfied by $\phi_t$. Therefore, after $q$ steps, the above procedure provides a subpartition $Q_1,\dots,Q_q$ of $\bR$ satisfying $\psi(I)=\phi_q(I)=\Phi(\bigcup_{i\in I}Q_i)$ for $I\subseteq[q]$. This subpartition, however, can be extended to a proper partition of $\bR$ by adding $\fr^{-1}(L_i)\setminus (\bigcup_{j=1}^{q} Q_j)$ to $Q_i$ for $i\in[q]$. This does not change the function $\phi_q$, concluding the proof of the theorem.
\end{proof}

\section{Conclusions and open problems}
\label{sec:open}

In this paper, we introduce the notion of coupling for set functions, providing a novel operation that combines two set functions into a single one. The proposed framework fills the gap arising from the fact that matroids do not always have a tensor product, and provides a construction that is applicable in every case while still possessing properties similar to those of the tensor product. As applications, we gave new necessary conditions for the representability of matroids, and established the existence of universal functions from which any increasing submodular function on a finite ground set can be obtained by taking quotients. 

We close the paper by mentioning some open problems:

\begin{enumerate}\itemsep0em
    \item Theorem~\ref{thm:poly} shows that two $2$-alternating functions always admit a $2$-alternating coupling, while Corollary~\ref{cor:inf} provides an analogous result for coverage functions. This leads to a natural question:  Do two $k$-alternating functions always have a $k$-alternating coupling?
    \item We have seen that the common information property is a necessary condition for linearity and in turn implies Ingleton's inequality. Additionally, we showed that having a tensor product with the uniform matroid $U_{2,3}$ is also a necessary condition for linearity, and it too implies Ingleton's inequality. It remains an intriguing open problem to clarify how these two necessary conditions are related.
    \item The proofs of Theorems~\ref{thm:submod} and~\ref{thm:poly} provided constructions for coupling two submodular functions and two polymatroid functions, respectively. However, it is not difficult to find examples showing that not all couplings can be obtained through these constructions. The description of the convex hull of couplings remains an interesting open problem.  
\end{enumerate}

\section*{Acknowledgments} 

The authors are grateful to Miklós Abért, Márton Borbényi, Máté András Pálfy and László Márton Tóth for helpful discussions. András Imolay was supported by the Rényi Doctoral Fellowship of the Rényi Institute, by the EKÖP-24 University Excellence Scholarship Program and by the EKÖP-KDP-25 University Research Scholarship Program (Cooperative Doctoral Program) of the Ministry for Culture and Innovation, funded from the National Research, Development and Innovation Fund. The research received further support from the Lend\"ulet Programme of the Hungarian Academy of Sciences (LP2021-1/2021), from the Ministry of Innovation and Technology of Hungary from the National Research, Development and Innovation Fund (ADVANCED 150556, 153096 and ELTE TKP 2021-NKTA-62), and from the Dynasnet European Research Council Synergy project (ERC-2018-SYG 810115). This work was supported in part by EPSRC grant EP/X030989/1.

\paragraph{AI declaration.} 
The AI tools used in this work, namely GPT-5.5, GPT-5.6 Sol, and Claude Opus 4.7 and 5.5, were employed solely for proofreading and grammar checking of the manuscript.
The text was fully reviewed and verified by the authors, who take full responsibility for the final content.

\paragraph{Data availability.}  
No data are associated with this article. Data sharing is not applicable to this article.

\bibliographystyle{amsplain}

\begin{aicauthors}
\begin{authorinfo}[kristof]
  Kristóf Bérczi\\
  MTA-ELTE Matroid Optimization Research Group and HUN-REN–ELTE Egerváry Research Group, Department of Operations Research, ELTE Eötvös Loránd University, and HUN-REN Alfréd Rényi Institute of Mathematics\\
  Budapest, Hungary\\
  kristof\imagedot{}berczi\imageat{}ttk\imagedot{}elte\imagedot{}hu
\end{authorinfo}
\begin{authorinfo}[bogi]
  Boglárka Gehér\\
  Department of Applied Analysis and Computational Mathematics, ELTE Eötvös Loránd University, and HUN-REN Alfréd Rényi Institute of Mathematics\\
  Budapest, Hungary\\
  bogigeher\imageat{}gmail\imagedot{}com
\end{authorinfo}
\begin{authorinfo}[andris]
  András Imolay\\
  Department of Operations Research, ELTE Eötvös Loránd University\\
  Budapest, Hungary\\
  andras\imagedot{}imolay\imageat{}ttk\imagedot{}elte\imagedot{}hu
\end{authorinfo}
\begin{authorinfo}[laci]
  László Lovász\\
  HUN-REN Alfréd Rényi Institute of Mathematics\\
  Budapest, Hungary\\
  laszlo\imagedot{}lovasz\imageat{}ttk\imagedot{}elte\imagedot{}hu
\end{authorinfo}
\begin{authorinfo}[balazs]
  Balázs Maga\\
  HUN-REN Alfréd Rényi Institute of Mathematics\\
  Budapest, Hungary\\
  magab\imageat{}renyi\imagedot{}hu
\end{authorinfo}
\begin{authorinfo}[tamas]
  Tamás Schwarcz\\
  ELTE Eötvös Loránd University and HUN-REN–ELTE Egerváry Research Group\\
  Budapest, Hungary\\
  \emph{Current address:} Hertz Chair for Algorithms and Optimization, University of Bonn, Bonn, Germany\\
  tamas\imagedot{}schwarcz\imageat{}uni-bonn\imagedot{}de
\end{authorinfo}
\end{aicauthors}

\end{document}